\documentclass[11pt,leqno,oneside]{amsart}
\usepackage{layout}
\usepackage{mathrsfs,dsfont}
\usepackage{amsmath,amstext,amsthm,amssymb,bbm,color}
\usepackage{mathtools}
\usepackage{enumerate}
\usepackage{comment}
\usepackage{charter}
\usepackage[nopatch=item]{microtype}
\usepackage{typearea}
\usepackage{pdfsync}
\usepackage[width=6.5in,height=8.7in]{geometry}

\usepackage{amsfonts}

\usepackage[utf8]{inputenc}

\def\bt{\begin{thm}}
\def\et{\end{thm}}
\def\bl{\begin{lem}}
\def\el{\end{lem}}
\def\bd{\begin{defn}}
\def\ed{\end{defn}}
\def\bc{\begin{cor}}
\def\ec{\end{cor}}
\def\bp{\begin{proof}}
\def\ep{\end{proof}}
\def\br{\begin{rem}}
\def\er{\end{rem}}

\newtheorem{thm}{Theorem}[section]
\newtheorem{prop}[thm]{Proposition}
\newtheorem{lem}[thm]{Lemma}
\newtheorem{defn}[thm]{Definition}
\newtheorem{example}[thm]{Example}
\newtheorem{rem}[thm]{Remark}
\newtheorem{cor}[thm]{Corollary}

\numberwithin{equation}{section}

\newcommand{\C}{\mathbb{C}}
\newcommand{\R}{\mathbb{R}}
\newcommand{\E}{\mathbb{E}}

\newcommand{\bthm}{\begin{thm}}
\newcommand{\ethm}{\end{thm}}
\newcommand{\bstp}{\begin{stp}}
\newcommand{\estp}{\end{stp}}
\newcommand{\blemma}{\begin{lemma}}
\newcommand{\elemma}{\end{lemma}}
\newcommand{\bprop}{\begin{prop}}
\newcommand{\eprop}{\end{prop}}
\newcommand{\bpf}{\begin{pf}}
\newcommand{\epf}{\end{pf}}
\newcommand{\bdefn}{\begin{defn}}
\newcommand{\edefn}{\end{defn}}
\newcommand{\brk}{\begin{rmrk}}
\newcommand{\erk}{\end{rmrk}}
\newcommand{\bcrl}{\begin{crl}}
\newcommand{\ecrl}{\end{crl}}
\newcommand{\Log}{\mathrm{Log}}
\newcommand{\Z}{\mathbb{Z}}
\newcommand{\PSH}{\mathrm{PSH}}

\title[Zeros of random $P$-polynomials]{Zeros of random $P$-polynomials in $\C^d$ with exponential  profiles}

\address{Faculty of Engineering and Natural Sciences, Sabanc{\i} University, \.{I}stanbul, Turkey}
\email{tbayraktar@sabanciuniv.edu}
\email{afrimbojnik@sabanciuniv.edu}

\date{\today}


\begin{document}

\author{Turgay Bayraktar \and Afrim Bojnik}
\thanks{T. Bayraktar is partially supported by T\"{U}B\.{I}TAK grant ARDEB-1001/124F37}
\thanks{A.\ Bojnik is partially supported by the Tosun Terzio\u{g}lu Chair Postdoctoral Fellowship.}

\begin{abstract}
We study random multivariate $P$-polynomials in $\C^d$ with monomial supports constrained to $nP\cap\Z_+^d$ for a convex body $P\subset\R_+^d$, and deterministic coefficients admitting a uniform exponential profile $f$ on $P$. Assuming the tail condition $\mathbb{P}(\log(1+|\xi_0|)>t)=o(t^{-d})$ on the i.i.d.\ complex coefficients, we prove that the normalized potentials $\frac1n\log|\mathbf{P}_n|$ converge in probability in $L^1_{\rm loc}(\C^d)$ to a deterministic toric plurisubharmonic function $\Phi_{P,f}$, and consequently the normalized zero currents $\frac1n[Z_{\mathbf{P}_n}]$ converge weakly to the closed positive $(1,1)$-current $dd^c\Phi_{P,f}$. Under the stronger logarithmic moment assumption $\E[(\log(1+|\xi_0|))^d]<\infty$, we prove almost sure weak convergence of the zero currents along the full sequence for $d>2$, and along sparse subsequences for $d \le 2$. On $(\C^*)^d$, the limiting potential is given by $\Phi_{P,f}(z)=I_{P,f}(\Log z)$, where $I_{P,f}$ is the Legendre--Fenchel transform of the profile over $P$ and $\Log (z)=(\log|z_1|,\dots,\log|z_d|)$. These results extend the exponential-profile mechanism of Kabluchko and Zaporozhets from one complex variable to the genuinely multivariate $P$-polynomial setting under relaxed probabilistic assumptions, directly connecting random zero hypersurfaces with convex-analytic data determined by $(P,f)$.\end{abstract}

\maketitle

%%%%%%%%%%%%%%%%%%%%%%%%%%%%%%%%%%%%%%%%%%%%%%%%%%%%%%%%%%%%%%%%%%%%%%%%%%%%%%%

\section{Introduction}

The asymptotic distribution of zeros of random holomorphic objects is a classical topic at the interface of probability, complex analysis, and geometry. In this paper we study random $P$-polynomials in $\C^d$ whose monomial supports lie in $nP\cap \Z_+^d$, where $P\subset \R_+^d$ is a convex body, and whose deterministic coefficient array admits a uniform exponential profile on $P$. Our aim is to identify the limiting logarithmic potential and the limiting zero current in terms of the deterministic data of the model, namely the support body $P$ and the coefficient profile.

The one-variable theory already contains two themes that are relevant to the present work. For classical Kac polynomials, Ibragimov and Zaporozhets \cite{IZ13} established the sharp logarithmic moment condition governing the asymptotic zero distribution. Kabluchko and Zaporozhets \cite{KZ14} later considered random analytic functions whose deterministic coefficients admit an exponential profile and proved, under the condition $\E[\log(1+|\xi_0|)]<\infty$, that the normalized zero measures converge in probability to a deterministic limit described by the Legendre--Fenchel transform of the profile. Thus, in this coefficient-driven setting, the limiting zero distribution is determined by the asymptotic behavior of the deterministic coefficients.

Another line of work concerns random sums built from polynomial or section families chosen in advance by geometric, orthogonal, or potential-theoretic considerations. In one complex variable, Dauvergne \cite{Dau21} obtained sharp results for random sums of orthogonal and, more generally, asymptotically minimal polynomial families on a compact set in $\C$. In geometric settings, Shiffman and Zelditch \cite{ShiffmanZelditch99} proved equidistribution of zeros of random holomorphic sections of high powers of a positive line bundle. In polynomial settings in several complex variables, weighted extremal functions, Bergman-type asymptotics, and pluripotential theory describe the corresponding limits; see, for example, \cite{BL15,Bay16,BBL24,BDL24,BG}. In these works, the random object is built from a preassigned family, such as a sequence of asymptotically minimal polynomials, an orthonormal basis of a polynomial space, or an orthonormal basis of a space of holomorphic sections (see \cite{BCHM} and references therein).

In several complex variables, sparsity of the monomial support brings an additional geometric feature into the asymptotic problem. When the support is restricted to $nP\cap \Z_+^d$, the convex body $P$ plays an essential role. In the Gaussian setting, Shiffman and Zelditch \cite{ShiffmanZelditch04Newton} showed that prescribed Newton polytope support leads to toric asymptotics and to the appearance of allowed and forbidden regions for zeros. Bayraktar \cite{Bay17} studied random sparse Laurent polynomials with prescribed polytope support for a broad class of non-Gaussian coefficient distributions. The corresponding deterministic framework is provided by $P$-pluripotential theory \cite{BBL18}, which provides the relevant tools such as growth classes, extremal functions, and equilibrium measures for this setting.

The present paper brings together the concepts of sparse support and prescribed exponential coefficient asymptotics. More precisely, the random object is built from a scaled sparse monomial array of the form $\{c_{\alpha,n}z^\alpha\}_{\alpha\in nP\cap \Z_+^d}$, where the deterministic coefficient array satisfies a uniform exponential profile $f$ on $P$ in the sense of \eqref{eq:profile}. What distinguishes this setting is that the asymptotic information is imposed directly on the coefficient array, rather than being derived from a chosen family arising from orthogonal or potential-theoretic constructions. At the same time, the framework is broad enough to include many structured arrays as special cases. In particular, it includes orthonormal monomial arrays arising from torus-invariant measures and torus-invariant weights. In this way, the ensemble is determined by the pair $(P,f)$, where $P$ fixes the support geometry and $f$ fixes the exponential scale of the coefficients. The corresponding limit is described by a toric plurisubharmonic function $\Phi_{P,f}$ naturally associated with $(P,f)$ and obtained explicitly from the Legendre--Fenchel transform of the profile on $P$. From this point of view, the model is close in spirit to the exponential-profile framework of \cite{KZ14}, while the limiting object belongs naturally to the setting of $P$-pluripotential theory.

Our first result identifies the limiting potential explicitly. Assuming the tail condition \eqref{eq:logtail} for $\log(1+|\xi_0|)$, we show that the normalized logarithmic modulus converges in probability to $\Phi_{P,f}$ in $L^1_{\mathrm{loc}}(\C^d)$. It follows that the normalized zero currents converge in probability to the closed positive current $dd^c\Phi_{P,f}$ (see Theorem~\ref{thm:potentials} and Corollary~\ref{cor:currents}). In dimension $d=1$, this convergence-in-probability result is obtained under a tail assumption which is strictly weaker than the logarithmic moment condition considered in \cite{KZ14}. Under the stronger logarithmic moment condition \eqref{eq:logmoment}, we also prove almost sure convergence of the normalized zero currents to $dd^c\Phi_{P,f}$ (see Theorem~\ref{thm:main}). This almost sure convergence of the normalized zero currents holds along the full sequence when $d>2$, and along suitable sparse subsequences when $d\le 2$.

We also show in Section~\ref{sec:torus-orthopoly} that the present framework includes a natural weighted orthogonal case as a specialization. More precisely, for the orthonormal monomial basis associated with a general torus-invariant weight, we verify the profile assumption \eqref{eq:profile}. In the special case $P=\Sigma$, our general construction recovers the toric extremal current arising in the weighted setting and the corresponding equidistribution result studied in \cite{Bay19}.

Thus, in the sparse multivariate setting considered here, the global asymptotic zero distribution is determined explicitly by the convex analytic data $(P,f)$. In this way, our results show that the profile-driven universality mechanism developed in one complex variable extends naturally to a genuinely multivariate sparse setting governed by Newton polytope geometry and $P$-pluripotential theory.

\subsection*{Random $P$-polynomials with exponential profiles}
Fix $d\ge 1$. Let $(\Omega,\mathcal F,\mathbb P)$ be a probability space and let
$(\xi_\alpha)_{\alpha\in\Z_+^d}$ be an i.i.d.\ family of complex-valued random variables on $\Omega$. Assume that $\xi_0$  is non-degenerate (i.e. not almost surely constant) and satisfies the tail condition
\begin{equation}\label{eq:logtail}
\mathbb P\big(\log(1+|\xi_0|)>t\big)=o(t^{-d})\quad \text{as } t\to\infty .
\end{equation}
We will also consider the stronger moment condition for our almost sure convergence results (see, e.g., Theorem~\ref{thm:main}) 
\begin{equation}\label{eq:logmoment}
\E\Big[\big(\log(1+|\xi_0|)\big)^d\Big] < \infty.
\end{equation}
It is easy to see that, by Markov's inequality, the moment condition \eqref{eq:logmoment} implies the tail bound \eqref{eq:logtail}. 

Let $P\subset \R_+^d=[0,\infty)^d$ be a convex body (i.e., compact, convex, and with nonempty interior) such that $0\in P$,
and write $nP:=\{nt:t\in P\}$ for $n\in\mathbb{N}$.
For each $n\in\mathbb{N}$ consider the random $P$-polynomial
\begin{equation}\label{eq:Gn}
\mathbf{P}_n(z,\omega)
:=\sum_{\alpha\in nP\cap\Z_+^d}\xi_\alpha(\omega)\,c_{\alpha,n}\,z^\alpha,
\qquad z\in\C^d,\qquad z^\alpha:=z_1^{\alpha_1}\cdots z_d^{\alpha_d}.
\end{equation}
We assume that the deterministic coefficients admit an asymptotic exponential profile on $P$, i.e.,
there exists a continuous function $f:P\to(0,\infty)$ such that
\begin{equation}\label{eq:profile}
\lim_{n\to\infty}\ \sup_{\alpha\in nP\cap\Z_+^d}\ \left|\frac1n\log|c_{\alpha,n}|-\log f(\alpha/n)\right|=0.
\end{equation}
Equivalently, if $u:=-\log f\in \mathscr{C}(P)$, then
\begin{equation*}
\frac{1}{n}\log|c_{\alpha,n}|=-u(\alpha/n)+o(1)
\qquad\text{as } n\to\infty,
\end{equation*}
uniformly for $\alpha\in nP\cap\Z_+^d$.
Since $\log f$ is finite on $P$, \eqref{eq:profile} implies that for all sufficiently large $n$,
\begin{equation*}
c_{\alpha,n}\neq 0\qquad\text{for every }\alpha\in nP\cap\Z_+^d.
\end{equation*}

For $s,t\in\R^d$ we denote by $\langle s,t\rangle:=\sum_{j=1}^d s_j t_j$ the Euclidean pairing.
Define the (restricted) Legendre--Fenchel transform of $u$ over $P$ by
\begin{equation}\label{eq:IP}
I_{P,f}(s):=\sup_{t\in P}\big(\langle s,t\rangle-u(t)\big)
          =\sup_{t\in P}\big(\langle s,t\rangle+\log f(t)\big),\qquad s\in\R^d.
\end{equation}
Since $P$ is compact and $u\in \mathscr{C}(P)$, the function $I_{P,f}$ is finite-valued, convex, and continuous on $\R^d$ (see Lemma \ref{lem:IPf-basic}).

Introduce the logarithmic map $\Log:(\C^*)^d\to\R^d$,
\begin{equation*}
\Log(z):=(\log|z_1|,\dots,\log|z_d|),\qquad z\in(\C^*)^d.
\end{equation*}
On the complex torus $(\C^*)^d$, define
\begin{equation}\label{eq:PhiPf-torus}
\Phi_{P,f}(z):=I_{P,f}(\Log(z))
=\sup_{t\in P}\big\{\langle \Log(z),t\rangle+\log f(t)\big\},
\qquad z\in(\C^*)^d,
\end{equation}
where $\langle \Log(z),t\rangle=\sum_{j=1}^d t_j\log|z_j|$.
To interpret the limiting object globally on $\C^d$, we extend $\Phi_{P,f}$ across the coordinate hyperplanes by
upper semicontinuous regularization:
\begin{equation}\label{eq:PhiPf}
\Phi_{P,f}(z):=\limsup_{\substack{\zeta\to z\\ \zeta\in(\C^*)^d}} I_{P,f}(\Log(\zeta)),
\qquad z\in\C^d.
\end{equation}
Then $\Phi_{P,f}\in\PSH(\C^d)$ and \eqref{eq:PhiPf-torus} holds on $(\C^*)^d$
(for details see Lemma~\ref{lem:PhiPf-psh}). In particular, $dd^c\Phi_{P,f}$ is a well-defined closed positive $(1,1)$-current on $\C^d$.

Now let $[Z_{\mathbf{P}_n}]$ be the current of integration over the zero hypersurface of $\mathbf{P}_n$ (counted with multiplicities). On the event $\{\mathbf{P}_n\equiv 0\}$ we set $[Z_{\mathbf{P}_n}]:=0$ and $\frac1n\log|\mathbf{P}_n|:=0$. This convention does not affect any of our limit statements, since $\mathbf{P}_n\equiv 0$ occurs only finitely many times almost surely (see Lemma~\ref{lem:eventual-nonzero}).

\subsection*{Main results}
We first prove the convergence of the normalized logarithmic modulus in $L^1_{\mathrm{loc}}(\C^d)$ \emph{in probability} under the tail assumption \eqref{eq:logtail}, and then deduce the convergence of the associated normalized zero currents. 

\begin{thm}\label{thm:potentials}
Assume \eqref{eq:profile} and \eqref{eq:logtail}. Then, in probability,
\begin{equation}\label{eq:L1prob}
   \frac{1}{n}\log|\mathbf{P}_n| \longrightarrow \Phi_{P,f}
\end{equation}
in \(L^1_{\mathrm{loc}}(\C^d)\) as \(n\to\infty\).
\end{thm}

As a consequence, we obtain the following convergence result for the normalized zero currents.

\begin{cor}\label{cor:currents}
Assume \eqref{eq:profile} and \eqref{eq:logtail}. Then the normalized zero currents $\frac1n[Z_{\mathbf{P}_n}]$
converge to \(dd^c\Phi_{P,f}\) in each of the following senses:
\begin{itemize}
\item[(i)] \emph{(Almost sure convergence along further subsequences)}
For every subsequence \((n_k)_{k\ge1}\), there exists a further subsequence \((n_{k_j})_{j\ge1}\) such that, as \(j\to\infty\),
\[
\frac{1}{n_{k_j}}[Z_{\mathbf{P}_{n_{k_j}}}] \longrightarrow dd^c\Phi_{P,f}
\qquad\text{almost surely in the weak sense of currents}.
\]

\item[(ii)] \emph{(Convergence in probability in the weak topology of currents)}
For every open neighborhood \(\mathcal{O}\subset\mathcal{D}'^{\,1,1}(\C^d)\) of \(dd^c\Phi_{P,f}\) for the weak topology of currents,
\[
\mathbb{P}\!\left(\frac1n[Z_{\mathbf{P}_n}]\in\mathcal{O}\right)\xrightarrow[n\to\infty]{}1.
\]
\end{itemize}
\end{cor}

\medskip

Our next result establishes almost sure convergence of the normalized zero currents under the stronger logarithmic moment condition \eqref{eq:logmoment}.

\begin{thm}\label{thm:main}
Assume \eqref{eq:profile} and \eqref{eq:logmoment}.
\begin{enumerate}[(i)]
\item For every deterministic increasing sequence $(\ell_k)_{k\ge1}$ such that $\sum_{k=1}^\infty \ell_k^{-d/2}<\infty$
(for instance, $\ell_k\ge k^3$), we have, as $k\to\infty$,
\begin{equation}
\frac{1}{\ell_k}[Z_{\mathbf{P}_{\ell_k}}] \longrightarrow dd^c\Phi_{P,f}
\qquad\text{almost surely in the weak sense of currents}.
\end{equation}

\item If, in addition, $d>2$, then, as $n\to\infty$,
\begin{equation}\label{eq:main-as}
\frac{1}{n}[Z_{\mathbf{P}_n}] \longrightarrow dd^c\Phi_{P,f}
\qquad\text{almost surely in the weak sense of currents}.
\end{equation}
\end{enumerate}
\end{thm}

\medskip

In the one-dimensional case \(d=1\), the preceding results admit a particularly concrete interpretation. Since \(P\subset \R_{+}\) is then a compact interval containing \(0\), the random \(P\)-polynomials \(\mathbf P_n\) are ordinary one-variable random polynomials whose degree grows linearly with \(n\). Moreover, \(\Phi_{P,f}\) is a subharmonic function on \(\C\), and \(dd^c\Phi_{P,f}\) is a positive measure on \(\C\). Accordingly, \([Z_{\mathbf P_n}]\) is the zero counting measure of \(\mathbf P_n\), and the convergence of normalized zero currents becomes the convergence of normalized zero counting measures toward the deterministic limiting measure \(dd^c\Phi_{P,f}\). In particular, Theorem~\ref{thm:potentials} yields convergence in probability of the normalized logarithmic potentials, Corollary~\ref{cor:currents} gives convergence in probability of the normalized zero counting measures, and Theorem~\ref{thm:main}(i) yields almost sure convergence along deterministic sparse subsequences under the stronger logarithmic moment condition \eqref{eq:logmoment}.

\subsection*{Organization}
Section~\ref{sec:background} collects the deterministic and probabilistic preliminaries used throughout the paper.
In Section~\ref{sec:convprob}, we prove the convergence-in-probability statements, specifically Theorem~\ref{thm:potentials} and Corollary~\ref{cor:currents}.
In Section~\ref{sec:conv-as}, we prove the almost sure convergence statements of Theorem~\ref{thm:main}.
Finally, in Section~\ref{sec:torus-orthopoly}, we verify the profile assumption \eqref{eq:profile} for orthonormal monomials associated with torus-invariant weighted measures. By specializing this to the standard simplex $P=\Sigma$, we recover the limiting zero current for the toric setting originally established in \cite{Bay19}.
\section{Background}\label{sec:background}

We collect deterministic and probabilistic ingredients used throughout the paper. On the deterministic side, we recall
the convex-analytic data attached to a convex body $P\subset\R_+^d$ and a continuous profile $f$, and the associated
toric plurisubharmonic potential $\Phi_{P,f}$ on $\C^d$. On the probabilistic side, we recall the notion of convergence in probability used later, together with concentration
functions and the Kolmogorov--Rogozin small-ball inequality.

% ------------------------------------------------------------
\subsection{Deterministic preliminaries}\label{subsec:PhiPf-background}

Let $\R_+^d:=[0,\infty)^d$ and let $P\subset\R_+^d$ be a convex body (i.e., $P$ is compact, convex, and has nonempty interior).
Throughout, we assume that $0\in P$. For $n\in\mathbb{N}$ we write $nP:=\{nt:\ t\in P\}$ for the dilates of $P$ and consider the lattice set $nP\cap\Z_+^d$,
where $\Z_+^d:=\{0,1,2,\dots\}^d$. 

For the sake of completeness, we state and prove the following lemma on the asymptotic growth of the cardinality of the lattice set, which will be used later.

\begin{lem}\label{lem:lattice-growth}
There exist constants $c_P,C_P>0$ and $n_P\in\mathbb{N}$ such that for all $n\ge n_P$,
\begin{equation*}
c_P\,n^d \ \le\ |nP\cap\Z_+^d|\ \le\ C_P\,n^d.
\end{equation*}
In particular, $\frac1n\log|nP\cap\Z_+^d|\to 0$ as $n\to\infty$.
\end{lem}

\begin{proof}
Since $P$ is compact, there exists $M>0$ such that $P\subset[0,M]^d$, hence
\begin{equation*}
|nP\cap\Z_+^d|
\le |[0,nM]^d\cap\Z^d|
\le (\lfloor nM\rfloor+1)^d \le C_P n^d.
\end{equation*}
For the lower bound, choose $t^{(0)}\in \mathrm{int}(P)$. Since $\mathrm{int}(P)$ is open, there exists $a>0$ such that
$t^{(0)}+[0,a]^d\subset P$, hence $nt^{(0)}+[0,an]^d\subset nP$.
For any $x\in\R$ and $L\ge 2$, the interval $[x,x+L]$ contains at least $\lfloor L\rfloor-1$ integers; applying this
coordinatewise yields
\begin{equation*}
|(nt^{(0)}+[0,an]^d)\cap\Z^d|
\ge (\lfloor an\rfloor-1)^d \ge c_P n^d
\end{equation*}
for all sufficiently large $n$. The final claim follows since $\log(n^d)/n\to 0$.
\end{proof}

\subsubsection*{Canonical potential}
For $s,t\in\R^d$ we write $\langle s,t\rangle:=\sum_{j=1}^d s_j t_j$. The \emph{support function of $P$} is given by
\begin{equation*}
h_P(s):=\sup_{t\in P}\langle s,t\rangle,\qquad s\in\R^d,
\end{equation*}
which is finite and convex. Moreover, $h_P$ uniquely determines $P$ via the supporting half-space representation
\begin{equation*}
P=\big\{t\in\R^d:\ \langle s,t\rangle\le h_P(s)\ \text{for all } s\in\R^d\big\},
\end{equation*}
a direct consequence of the separation theorem for closed convex sets (see, e.g., \cite[Ch.~11]{Rockafellar70}). 

On $(\C^*)^d$ we use the logarithmic map $\Log:(\C^*)^d\to\R^d$, given by
\begin{equation*}
\Log(z):=(\log|z_1|,\dots,\log|z_d|)\in\R^d,
\end{equation*}
and define the \emph{logarithmic indicator function} associated with $P$ by
\begin{equation}\label{eq:HP}
H_P(z):=h_P(\Log(z))
=\sup_{t\in P}\sum_{j=1}^d t_j\log|z_j|,
\qquad z\in(\C^*)^d.
\end{equation}
We extend $H_P$ to $\C^d$ by its upper semicontinuous regularization, using the convention $\log 0=-\infty$.
Since $0\in P$, the choice $t=0$ shows that $H_P\ge 0$ on $(\C^*)^d$, hence also on $\C^d$.

A function on $(\C^*)^d$ is called \emph{$d$-circled} (or \emph{toric}) if it depends only on $(|z_1|,\dots,|z_d|)$; in particular, $H_P$ is $d$-circled.

\begin{example}\label{rem:HP-examples}
For the standard simplex $\Sigma=\{t\in\R_+^d:\ \sum_{j=1}^d t_j\le1\}$, one has
\begin{equation*}
H_\Sigma(z)=\max\{0,\log|z_1|,\dots,\log|z_d|\}=\max_{1\le j\le d}\log^+|z_j|.
\end{equation*}
For the cube $Q=[0,1]^d$, one has $H_Q(z)=\sum_{j=1}^d \log^+|z_j|$. If $P$ is a non-degenerate convex polytope (that is, the convex hull of a finite subset of $\mathbb{Z}^d_+$ in $\mathbb{R}^d_+$ with $\mathrm{int}(P)\neq \emptyset$), then $h_P(s)=\max_{v\in\mathrm{Vert}(P)}\langle s,v\rangle$ and hence on $(\C^*)^d$,
\begin{equation*}
H_P(z)=\max_{v\in\mathrm{Vert}(P)}\sum_{j=1}^d v_j\log|z_j|,
\end{equation*}
with the extension to $\C^d$ obtained by upper semicontinuous regularization.
\end{example}

\medskip

Now let $f:P\to(0,\infty)$ be continuous and set $u:=-\log f\in \mathscr{C}(P)$.
The \emph{restricted Legendre--Fenchel transform} of $u$ over $P$ is given by
\begin{equation}\label{eq:IP-background}
I_{P,f}(s):=\sup_{t\in P}\big(\langle s,t\rangle-u(t)\big)
=\sup_{t\in P}\big(\langle s,t\rangle+\log f(t)\big),
\qquad s\in\R^d.
\end{equation}

\begin{lem}\label{lem:IPf-basic}
The function $I_{P,f}:\R^d\to\R$ is finite-valued, convex, and globally Lipschitz. For each $s\in\R^d$, the supremum in
\eqref{eq:IP-background} is attained. Moreover,
\begin{equation}\label{eq:IP-vs-hP}
h_P(s)+\inf_{P}\log f\ \le\ I_{P,f}(s)\ \le\ h_P(s)+\sup_{P}\log f,
\qquad s\in\R^d,
\end{equation}
hence $I_{P,f}-h_P$ is bounded on $\R^d$, and in particular $I_{P,1}=h_P$.
\end{lem}

\begin{proof}
Since $\log f\in \mathscr{C}(P)$ and $P$ is compact, $m_f:=\inf_{P}\log f$ and $M_f:=\sup_{P}\log f$ are finite.
For fixed $s\in\R^d$, the map $t\mapsto \langle s,t\rangle+\log f(t)$ is continuous on $P$, hence attains its maximum, which
gives finiteness of $I_{P,f}(s)$ and attainment of the supremum.

Convexity follows from the fact that for each $t\in P$ the function $s\mapsto \langle s,t\rangle+\log f(t)$ is affine, and $I_{P,f}$ is the
pointwise supremum of these affine functions.

For the Lipschitz bound (with respect to the Euclidean norm), set $R_P:=\sup_{t\in P}\|t\|<\infty$. Then for any $s,s'\in\R^d$,
\begin{equation*}
I_{P,f}(s)-I_{P,f}(s')
=\sup_{t\in P}\big(\langle s,t\rangle+\log f(t)\big)-\sup_{t\in P}\big(\langle s',t\rangle+\log f(t)\big)
\le \sup_{t\in P}\langle s-s',t\rangle
\le R_P\|s-s'\|.
\end{equation*}
Exchanging $s$ and $s'$ yields $|I_{P,f}(s)-I_{P,f}(s')|\le R_P\|s-s'\|$.

Finally, $m_f\le \log f(t)\le M_f$ for all $t\in P$ implies
\begin{equation*}
\sup_{t\in P}\big(\langle s,t\rangle+m_f\big)\le I_{P,f}(s)\le \sup_{t\in P}\big(\langle s,t\rangle+M_f\big),
\end{equation*}
i.e., $h_P(s)+m_f\le I_{P,f}(s)\le h_P(s)+M_f$, which is \eqref{eq:IP-vs-hP}. If $f\equiv 1$, then $m_f=M_f=0$ and $I_{P,1}=h_P$.
\end{proof}

\medskip

On the complex torus $(\C^*)^d$ we define the \emph{canonical potential} as in the introduction by 
\begin{equation}\label{eq:PhiPf-torus-background}
\Phi_{P,f}(z):=I_{P,f}(\Log(z))
=\sup_{t\in P}\Big\{\sum_{j=1}^d t_j\log|z_j|+\log f(t)\Big\},
\qquad z\in(\C^*)^d,
\end{equation}
and extend it to $\C^d$ by upper semicontinuous regularization:
\begin{equation}\label{eq:PhiPf-usc-background}
\Phi_{P,f}(z):=\limsup_{\substack{\zeta\to z\\ \zeta\in(\C^*)^d}} I_{P,f}(\Log\zeta),
\qquad z\in\C^d.
\end{equation}

\begin{lem}\label{lem:PhiPf-psh}
The function $\Phi_{P,f}$ belongs to $\PSH(\C^d)$ and is $d$-circled. Moreover,
\begin{equation}\label{eq:Phi-vs-H}
H_P(z)+\inf_{P}\log f\ \le\ \Phi_{P,f}(z)\ \le\ H_P(z)+\sup_{P}\log f,
\qquad z\in\C^d.
\end{equation}
\end{lem}

\begin{proof}
For each $t\in P$ define
\begin{equation*}
\psi_t(z):=\sum_{j=1}^d t_j\log|z_j|+\log f(t),\qquad z\in\C^d,
\end{equation*}
Since $\log|z_j|$ is plurisubharmonic and $t_j\ge0$ for $t\in P\subset\R_+^d$, it follows that $\psi_t\in\PSH(\C^d)$ for every $t\in P$.
Observe that the family $(\psi_t)_{t\in P}$ is locally uniformly bounded above.
Indeed, let $K\subset\C^d$ be compact and set $M_j:=\sup_{z\in K}\log^+|z_j|<\infty$. Then for $z\in K$,
\begin{equation*}
\sum_{j=1}^d t_j\log|z_j|
\le \sum_{j=1}^d t_j\log^+|z_j|
\le \Big(\max_{1\le j\le d} M_j\Big)\sum_{j=1}^d t_j.
\end{equation*}
Since $P$ is compact, $\sup_{t\in P}\sum_{j=1}^d t_j<\infty$, and since $\log f$ is continuous on $P$, $\sup_P\log f<\infty$.
Hence $\sup_{t\in P}\sup_{z\in K}\psi_t(z)<\infty$.

Define $\Psi(z):=\sup_{t\in P}\psi_t(z)$ on $\C^d$ and let $\Psi^*$ be its upper semicontinuous regularization.
By standard properties of upper envelopes, $\Psi^*\in\PSH(\C^d)$ (see, e.g., \cite[Ch.\,I]{Demailly12} or \cite[Ch.\,III]{Hormander90}).

On $(\C^*)^d$ we have
\begin{equation*}
\Psi(z)=\sup_{t\in P}\big\{\langle\Log(z),t\rangle+\log f(t)\big\}=I_{P,f}(\Log z).
\end{equation*}
Therefore, by \eqref{eq:PhiPf-usc-background}, for every $z\in\C^d$, we have
\begin{equation*}
\Phi_{P,f}(z)=\limsup_{\substack{\zeta\to z\\ \zeta\in(\C^*)^d}}I_{P,f}(\Log\zeta)
=\limsup_{\substack{\zeta\to z\\ \zeta\in(\C^*)^d}}\Psi(\zeta)=\Psi^*(z),
\end{equation*}
where the last equality follows because one always has
$\limsup_{\zeta\to z,\ \zeta\in(\C^*)^d}\Psi(\zeta)\le \Psi^*(z)$ (since $(\C^*)^d\subset\C^d$), while conversely one may choose
$w_k\to z$ in $\C^d$ with $\Psi(w_k)\to \Psi^*(z)$ and replace any zero coordinate $w_{k,j}=0$ by $1/k$ to obtain
$\zeta_k\in(\C^*)^d$ with $\zeta_k\to z$ and $\Psi(\zeta_k)\ge \Psi(w_k)$ (as $t_j\ge0$ and $\log|0|=-\infty$), which yields the reverse inequality.
Hence $\Phi_{P,f}\in\PSH(\C^d)$.

Since each $\psi_t$ depends only on $(|z_1|,\dots,|z_d|)$, the same is true of $\Phi_{P,f}$, so $\Phi_{P,f}$ is $d$-circled.
Finally, applying \eqref{eq:IP-vs-hP} with $s=\Log(z)$ yields \eqref{eq:Phi-vs-H} on $(\C^*)^d$, and taking
$\limsup_{\zeta\to z,\ \zeta\in(\C^*)^d}$ extends the inequalities to $\C^d$.
\end{proof}

\medskip

We recall the classical Lelong growth classes on $\C^d$:
\begin{equation*}
\mathcal{L}(\C^d):=\Big\{v\in\PSH(\C^d):\ v(z)\le \log^+\|z\|+O(1)\ \text{as }\|z\|\to\infty\Big\},
\end{equation*}
\begin{equation*}
\mathcal{L}^+(\C^d):=\Big\{v\in \mathcal{L}(\C^d):\ v(z)\ge \log^+\|z\|+O(1)\ \text{as }\|z\|\to\infty\Big\}.
\end{equation*}
In the $P$-setting, one uses the corresponding $P$-growth classes defined relative to $H_P$:
\begin{equation*}
\mathcal{L}_P(\C^d):=\Big\{v\in\PSH(\C^d):\ v(z)-H_P(z)=O(1)\ \text{as }\|z\|\to\infty\Big\},
\end{equation*}
\begin{equation*}
\mathcal{L}_{P}^+(\C^d):=\Big\{v\in \mathcal{L}_P(\C^d):\ v(z)\ge H_P(z)+C_v\ \text{on }\C^d\Big\}.
\end{equation*}
Clearly, by Lemma~\ref{lem:PhiPf-psh}, $\Phi_{P,f}\in \mathcal{L}_{P}^+(\C^d)$. Observe that when $P=\Sigma$, one recovers the classical logarithmic regime since
$H_\Sigma(z)=\max_{1\le j\le d}\log^+|z_j|=\log^+\|z\|+O(1)$.

\begin{rem}[Comparison with the simplex growth]\label{rem:Sigma-kP}
Set $A_P:=\max_{t\in P}|t|_1=\max_{t\in P}\sum_{j=1}^d t_j.$
Then $P\subset A_P\,\Sigma$, hence
\[
H_P(z)\le A_P\,H_\Sigma(z)\qquad \text{for all } z\in\C^d,
\]
and therefore every $u\in \mathcal{L}_P(\C^d)$ has $O(\log\|z\|)$ growth. More precisely,
\[
u(z)\le A_P\,H_\Sigma(z)+O(1)=A_P\log^+\|z\|+O(1)\qquad \text{as }\|z\|\to\infty.
\]
Moreover, since $0\in P$, we have $H_P\ge 0$ on $\C^d$. In particular, if $u\in \mathcal{L}_{P}^+(\C^d)$, then $u\ge H_P+C_u\ge C_u$ on $\C^d,$
so $u$ never takes the value $-\infty$. In the present paper, we work intrinsically with $H_P$ and therefore do not impose auxiliary geometric assumptions such as $\Sigma\subset kP$, which are used in parts of the literature to force comparability with classical logarithmic indicators (see, e.g., \cite{BBL18}).
\end{rem}

\subsection*{Currents and Poincar\'e--Lelong formula}
We write $\mathcal{D}^{p,q}(\C^d)$ for the space of smooth compactly supported $(p,q)$-forms and
$\mathcal{D}'^{p,q}(\C^d)$ for the space of $(p,q)$-currents, i.e., continuous linear functionals on
$\mathcal{D}^{d-p,d-q}(\C^d)$, endowed with its weak topology.
Thus $T_n\to T$ in $\mathcal{D}'^{p,q}(\C^d)$ means that
\[
\langle T_n,\psi\rangle\to\langle T,\psi\rangle
\qquad\text{for every }\psi\in\mathcal{D}^{d-p,d-q}(\C^d).
\]

Throughout we adopt the normalization $dd^{c}:=\frac{i}{\pi}\partial\bar\partial$ and will repeatedly use that
if $v_n\to v$ in $L^1_{\mathrm{loc}}(\C^d)$, then $dd^c v_n\to dd^c v$ weakly in
$\mathcal{D}'^{1,1}(\C^d)$.
Indeed, for every test form $\psi\in\mathcal{D}^{d-1,d-1}(\C^d)$,
\[
\langle dd^c(v_n-v),\psi\rangle=\langle v_n-v,dd^c\psi\rangle.
\]
Since $dd^c\psi$ is a smooth compactly supported $(d,d)$-form, it follows that
\[
|\langle v_n-v,dd^c\psi\rangle|
\le C_\psi \|v_n-v\|_{L^1(\operatorname{supp}\psi)}
\]
for some constant $C_\psi>0$, which tends to zero as $n\to\infty$.

Moreover, for a nonzero holomorphic function $F$ on $\mathbb{C}^d$, the Poincar\'e--Lelong formula gives \begin{equation*} dd^c\log|F|=[Z_F]\qquad\text{in }\mathcal{D}'^{1,1}(\mathbb{C}^d), \end{equation*}
where $[Z_F]$ denotes the current of integration over the zero hypersurface of $F$, counted with multiplicities.
(When $F\not\equiv 0$, one has $\log|F|\in L^1_{\mathrm{loc}}(\C^d)$.)

\medskip

In what follows, we will repeatedly use standard compactness and uniqueness properties of plurisubharmonic functions to establish $L^1_{\mathrm{loc}}$ convergence of the normalized logarithmic potentials and weak convergence of their associated currents.

\begin{lem}\label{lem:psh-tools}
Let $(u_j)$ be a sequence of plurisubharmonic functions on $\mathbb{C}^d$ that is locally uniformly bounded above. Then:
\begin{enumerate}
    \item[\textup{(i)}] Either $u_j\to -\infty$ locally uniformly on $\mathbb{C}^d$, or $(u_j)$ admits a subsequence that converges in $L^1_{\mathrm{loc}}(\mathbb{C}^d)$ to a plurisubharmonic function.
    \item[\textup{(ii)}] If $u,v\in\operatorname{PSH}(\mathbb{C}^d)$ and $u=v$ Lebesgue-a.e.\ on $\mathbb{C}^d$, then $u\equiv v$ on $\mathbb{C}^d$.
\end{enumerate}
\end{lem}

\begin{proof}
For \textup{(i)}, see \cite[Prop.~5.9]{Demailly12}. The identity in \textup{(ii)} follows from the standard regularization of plurisubharmonic functions; see \cite[Thm.~5.5]{Demailly12} or \cite[\S\,2.6]{Klimek91}.
\end{proof}

\begin{cor}\label{cor:psh-ae-to-L1}
Let $(u_j)$ be plurisubharmonic on $\mathbb{C}^d$ and locally uniformly bounded above. 
Assume that $u_j\to u$ Lebesgue-a.e.\ on $\mathbb{C}^d$ for some plurisubharmonic function $u \not\equiv -\infty$.
Then $u_j\to u$ in $L^1_{\mathrm{loc}}(\mathbb{C}^d)$ and hence $dd^c u_j\to dd^c u$ weakly as currents.
\end{cor}

\begin{proof}
Let $(u_{j_k})$ be any subsequence. Because $u \not\equiv -\infty$, the function $u$ is finite Lebesgue-almost everywhere on $\mathbb{C}^d$.
Since $u_j \to u$ almost everywhere, the sequence $(u_j)$ cannot converge to $-\infty$ locally uniformly.

Thus, by Lemma~\ref{lem:psh-tools}\textup{(i)}, $(u_{j_k})$ admits a further subsequence $(u_{j_{k_\ell}})$ converging in $L^1_{\mathrm{loc}}(\mathbb{C}^d)$ to some plurisubharmonic function $v$.
After extracting once more, we may assume that $u_{j_{k_\ell}}\to v$ Lebesgue-almost everywhere on $\mathbb{C}^d$.

Since we also have $u_{j_{k_\ell}}\to u$ Lebesgue-almost everywhere by assumption, it follows that $v=u$ almost everywhere.
By Lemma~\ref{lem:psh-tools}\textup{(ii)}, we conclude that $v \equiv u$.

Thus, every subsequence has a further subsequence converging to $u$ in $L^1_{\mathrm{loc}}(\mathbb{C}^d)$, which implies that the entire sequence converges: $u_j\to u$ in $L^1_{\mathrm{loc}}(\mathbb{C}^d)$.
The weak convergence $dd^c u_j\to dd^c u$ then follows from the continuity of $dd^c$ under $L^1_{\mathrm{loc}}$ convergence.
\end{proof}

% ------------------------------------------------------------

\subsection{Probabilistic preliminaries}\label{subsec:prob-prelim}

We recall the notion of convergence in probability that will be used for logarithmic potentials in $L^1_{\rm loc}(\C^d)$ and for currents in $\mathcal D'^{1,1}(\C^d)$.
Since $L^1_{\rm loc}(\C^d)$ is metrizable, convergence in probability there can be treated in the usual metric sense; one convenient metric inducing the standard $L^1_{\rm loc}$ topology is
\begin{equation*}
d_{L^1_{\rm loc}}(u,v)
:=\sum_{m=1}^\infty 2^{-m}\min\Big\{1,\ \|u-v\|_{L^1(B_m)}\Big\},
\qquad 
B_m:=\{z\in\C^d:\ \|z\|\le m\}.
\end{equation*}
In what follows, when discussing convergence in probability in the metrizable space $L^1_{\mathrm{loc}}$, we will frequently use the following standard characterization, which we refer to as the \emph{subsequence principle}:

\emph{$X_n\to x$ in probability if and only if for every subsequence $(X_{n_k})$
there exists a further subsequence $(X_{n_{k_\ell}})$ such that
$X_{n_{k_\ell}}\to x$ almost surely.}

For currents in $\mathcal D'^{1,1}(\C^d)$ equipped with the weak topology of currents, the space is not metrizable in general, so we use the neighborhood formulation of convergence in probability recalled below.

\begin{rem}\label{rem:curr-prob}
Let $(\Omega,\mathcal F,\mathbb P)$ be a probability space, let $(S,\tau)$ be a topological space endowed with its Borel $\sigma$-algebra, let $X_n:\Omega\to S$ be $S$-valued random variables, and let $x\in S$ be deterministic. We say that
\begin{equation*}
X_n\to x \text{ in probability}
\qquad\Longleftrightarrow\qquad
\forall\ \text{open neighborhood }\mathcal O\ni x,\ \ \mathbb{P}(X_n\in\mathcal O)\to 1.
\end{equation*}
In metrizable spaces this agrees with the usual metric definition, and it is equivalent to the subsequence principle stated above.

The weak topology on $\mathcal D'^{1,1}(\C^d)$ is not metrizable in general. We therefore record both conclusions in Corollary~\ref{cor:currents}: (ii) is convergence in probability in the neighborhood sense for the weak topology of currents, while (i) is the corresponding almost sure subsequence convergence statement, deduced from Theorem~\ref{thm:potentials} using the continuity of the operator $dd^c:L^1_{\mathrm{loc}}(\C^d)\to\mathcal D'^{1,1}(\C^d).$
\end{rem}

\medskip

We next record the probabilistic tools used later to control logarithmic potentials and small-ball probabilities
for sums indexed by $nP\cap\Z_+^d$. 

For a complex-valued random variable $X$, define the \emph{concentration function}
\begin{equation}\label{eq:conc-fn}
\mathcal Q(X;r):=\sup_{w\in\C}\mathbb P\big(X\in B(w,r)\big),\qquad r>0,
\end{equation}
where $B(w,r)\subset\C$ is the Euclidean disk of center $w$ and radius $r$.
We call $X$ \emph{non-degenerate} if it is not almost surely constant; equivalently $\mathcal Q(X;r)<1$ for some $r>0$.
For $c\in\C$ and $a\in\C\setminus\{0\}$ one has
\begin{equation*}
\mathcal Q(X+c;r)=\mathcal Q(X;r),
\qquad
\mathcal Q(aX;r)=\mathcal Q\!\left(X;\frac{r}{|a|}\right),
\end{equation*}
and if $X,Y$ are independent then
\begin{equation}\label{eq:conc-basic-ineq}
\mathcal Q(X+Y;r)\le \mathcal Q(X;r),\qquad r>0.
\end{equation}

\begin{thm}[Kolmogorov--Rogozin inequality]\label{thm:KR}
There exists a constant $C_{\mathrm{KR}}>0$ such that for every $m\in\mathbb{N}$, every family of independent complex-valued random variables
$X_1,\dots,X_m$, and every $r>0$,
\begin{equation}\label{eq:KR}
\mathcal Q\Big(\sum_{j=1}^m X_j;\,r\Big)
\le
\frac{C_{\mathrm{KR}}}{\sqrt{\sum_{j=1}^m \bigl(1-\mathcal Q(X_j;r)\bigr)}}.
\end{equation}
\end{thm}

\begin{rem}\label{rem:KR}
Theorem~\ref{thm:KR} is the $\C$ (equivalently $\R^2$) case of the Kolmogorov--Rogozin inequality for concentration functions.
For references and related forms of the inequality, see Petrov~\cite[Ch.~II, \S2]{Petrov75}, Kesten~\cite{Kesten69},
and Ess\'een~\cite{Ess68}.
\end{rem}

\medskip

The deterministic tools above, together with Theorem~\ref{thm:KR}, will be combined with the coefficient assumptions
\eqref{eq:logtail} and \eqref{eq:logmoment} in Sections~\ref{sec:convprob} and~\ref{sec:conv-as} to control small-ball probabilities
for the random sums defining $\mathbf P_n$ and to obtain pointwise convergence of the normalized logarithmic potentials.

\section{Convergence in probability}\label{sec:convprob}

In this section we prove Theorem~\ref{thm:potentials} and Corollary~\ref{cor:currents}. Assuming the tail condition \eqref{eq:logtail}, we first establish probabilistic bounds on the random coefficients. Using standard compactness/uniqueness properties of plurisubharmonic functions, we then upgrade the pointwise convergence to $L^1_{\rm loc}(\C^d)$ convergence, and hence obtain weak convergence in probability of the normalized zero currents.

\subsection{Coefficient control under tail assumption}\label{subsec:prob-coeff}
We begin by deriving two elementary subexponential bounds on the random variables required for the subsequent proofs.

\begin{lem}\label{lem:xi-subexp-prob}
Assume \eqref{eq:logtail}. Then for every $\varepsilon>0$,
\begin{equation}\label{eq:max-xi-op}
\lim_{n \to \infty} \mathbb P\Big(\max_{\alpha\in nP\cap\Z_+^d}\log(1+|\xi_\alpha|)>\varepsilon n\Big) = 0.
\end{equation}
\end{lem}

\begin{proof}
Fix $\varepsilon>0$. Observe that
\begin{equation*}
\Big\{\max_{\alpha\in nP\cap\Z_+^d}\log(1+|\xi_\alpha|)>\varepsilon n\Big\}
=\bigcup_{\alpha\in nP\cap\Z_+^d}\Big\{\log(1+|\xi_\alpha|)>\varepsilon n\Big\}.
\end{equation*}
Hence, by the union bound and the identical distribution of $(\xi_\alpha)$,
\begin{equation*}
\mathbb P\Big(\max_{\alpha\in nP\cap\Z_+^d}\log(1+|\xi_\alpha|)>\varepsilon n\Big)
\le |nP\cap\Z_+^d|\,\mathbb P\big(\log(1+|\xi_0|)>\varepsilon n\big).
\end{equation*}
By Lemma~\ref{lem:lattice-growth}, $|nP\cap\Z_+^d|\le C_P n^d$ for all sufficiently large $n$, and therefore
\begin{equation*}
\mathbb P\Big(\max_{\alpha\in nP\cap\Z_+^d}\log(1+|\xi_\alpha|)>\varepsilon n\Big)
\le C_P n^d\,\mathbb P\big(\log(1+|\xi_0|)>\varepsilon n\big)
= C_P\varepsilon^{-d}\Big((\varepsilon n)^d\,\mathbb P(\log(1+|\xi_0|)>\varepsilon n)\Big).
\end{equation*}
The term in parentheses tends to $0$ as $n\to\infty$ by \eqref{eq:logtail}, which proves \eqref{eq:max-xi-op}.
\end{proof}

We will also need an almost sure bound along a sparse extraction from an arbitrary subsequence.

\begin{lem}\label{lem:good-subsequence}
Assume \eqref{eq:logtail}. Let $(n_k)_{k\ge1}$ be any increasing sequence of integers.
Then there exists a further subsequence $(m_j)_{j\ge1}$, with $m_j:=n_{k_j}$, such that for all $j \ge 1$,
\begin{equation}\label{eq:good-subseq-choice}
m_j\ge j^3
\quad\text{and}\quad
m_j^d\,\mathbb P\!\left(\log(1+|\xi_0|)>\frac{m_j}{j}\right)\le 2^{-j}.
\end{equation}
Consequently, there exists an event $\Omega_\xi$ with $\mathbb P(\Omega_\xi)=1$ such that for every
$\omega\in\Omega_\xi$,
\begin{equation}\label{eq:good-subseq-bc}
\max_{\alpha\in m_jP\cap\Z_+^d}\log(1+|\xi_\alpha(\omega)|)\le \frac{m_j}{j}
\end{equation}
for all sufficiently large $j$.
\end{lem}

\begin{proof}
Fix $j\ge 1$. The assumption \eqref{eq:logtail} implies that
$
\lim_{t \to \infty} t^d\mathbb P\big(\log(1+|\xi_0|)>t\big) = 0,
$
and hence
\begin{equation*}
\lim_{n \to \infty} n^d\,\mathbb P\!\left(\log(1+|\xi_0|)>\frac{n}{j}\right)
= j^d\lim_{n \to \infty} \Big(\Big(\frac{n}{j}\Big)^d\mathbb P\Big(\log(1+|\xi_0|)>\frac{n}{j}\Big)\Big) = 0.
\end{equation*}
Since $n_k\to\infty$, we can inductively choose $k_j>k_{j-1}$ large enough such that $m_j:=n_{k_j}\ge j^3$ and
$m_j^d\mathbb P(\log(1+|\xi_0|)>m_j/j)\le 2^{-j}$, proving \eqref{eq:good-subseq-choice}.

Using the union bound, the identical distribution of $(\xi_\alpha)$, and Lemma~\ref{lem:lattice-growth}, we have
\begin{align*}
\mathbb P\!\left(\max_{\alpha\in m_jP\cap\Z_+^d}\log(1+|\xi_\alpha|)>\frac{m_j}{j}\right)
&\le |m_jP\cap\Z_+^d|\,\mathbb P\!\left(\log(1+|\xi_0|)>\frac{m_j}{j}\right)\\
&\le C_P\,m_j^d\,\mathbb P\!\left(\log(1+|\xi_0|)>\frac{m_j}{j}\right)\\
&\le C_P\,2^{-j}.
\end{align*}
Since $\sum_{j=1}^\infty C_P 2^{-j}<\infty$,  Borel--Cantelli lemma yields \eqref{eq:good-subseq-bc}.
\end{proof}

\subsection{Pointwise convergence and the $L^1_{\rm loc}$ upgrade}\label{subsec:pointwise-L1}

First we prove pointwise convergence in probability of the normalized logarithmic potentials on $(\C^*)^d$.
\begin{thm}\label{thm:KZ-P-Theorem41}
Let $(\xi_{\alpha})_{\alpha\in\Z_+^d}$ be i.i.d.\ complex-valued non-degenerate random variables and assume \eqref{eq:profile}.
If \eqref{eq:logtail} holds, then for every $z\in(\C^*)^d$,
\begin{equation}\label{eq:pot-conv-prob}
\frac1n\log|\mathbf{P}_n(z)|\longrightarrow \Phi_{P,f}(z)\qquad\text{in probability as }n\to\infty.
\end{equation}
\end{thm}

\begin{proof}
Fix $z\in(\C^*)^d$ and set $s=\Log(z)=(\log|z_1|,\dots,\log|z_d|)\in\R^d$.
By the definition of $\Phi_{P,f}$ on $(\C^*)^d$ (cf.\ \eqref{eq:PhiPf-torus}), we have
\begin{equation*}
\Phi_{P,f}(z)=I_{P,f}(s),\qquad
I_{P,f}(s)=\sup_{t\in P}\big(\langle s,t\rangle+\log f(t)\big).
\end{equation*}
Set
\begin{equation*}
h_s(t):=\langle s,t\rangle+\log f(t)=\langle s,t\rangle-u(t),\qquad t\in P,
\end{equation*}
so that $I_{P,f}(s)=\sup_{t\in P}h_s(t)$. We also write $|\alpha|:=|\alpha|_1=\alpha_1+\cdots+\alpha_d$.

\smallskip\noindent\emph{Upper bound in probability.}
Assume \eqref{eq:logtail} and fix $\varepsilon>0$. Set
\begin{equation*}
E_n:=\Big\{\max_{\alpha\in nP\cap\Z_+^d}\log(1+|\xi_\alpha|)\le \varepsilon n\Big\}.
\end{equation*}
By Lemma~\ref{lem:xi-subexp-prob}, $\lim_{n \to \infty} \mathbb P(E_n^c) = 0$.
By the profile condition \eqref{eq:profile}, there exists $n_0=n_0(\varepsilon)$ such that for all $n\ge n_0$ and all $\alpha\in nP\cap\Z_+^d$,
\begin{equation*}
\frac1n\log|c_{\alpha,n}|\le \log f(\alpha/n)+\varepsilon.
\end{equation*}
For such $n$ and on $E_n$ we have $|\xi_\alpha|\le e^{\varepsilon n}$, hence for every $\alpha\in nP\cap\Z_+^d$,
\begin{equation*}
|\xi_\alpha c_{\alpha,n}z^\alpha|
\le \exp\Big(n\big(\langle \alpha/n,s\rangle+\log f(\alpha/n)+2\varepsilon\big)\Big)
\le \exp\big(n(\Phi_{P,f}(z)+2\varepsilon)\big).
\end{equation*}
Summing over $\alpha\in nP\cap\Z_+^d$ via the triangle inequality and using Lemma~\ref{lem:lattice-growth}, we obtain, for all large $n$ and on $E_n$,
\begin{equation*}
|\mathbf{P}_n(z)|\le |nP\cap\Z_+^d|\,\exp\big(n(\Phi_{P,f}(z)+2\varepsilon)\big)\le C_Pn^d\,\exp\big(n(\Phi_{P,f}(z)+2\varepsilon)\big).
\end{equation*}
Because $C_P n^d \le e^{\varepsilon n}$ for all sufficiently large $n$, we have
\begin{equation}\label{eq:upper-tail-prob}
\lim_{n \to \infty} \mathbb P\!\left(\frac1n\log|\mathbf{P}_n(z)|>\Phi_{P,f}(z)+3\varepsilon\right)
\le \lim_{n \to \infty} \mathbb P(E_n^c) = 0.
\end{equation}

\smallskip\noindent\emph{Lower bound in probability (small-ball estimate).}
Fix $\eta>0$. Since $h_s$ is continuous on the compact set $P$, there exists $t_0\in P$ with
$h_s(t_0)=I_{P,f}(s)$. Choose $\rho>0$ such that with $U:=P\cap B(t_0,\rho)$ one has
\begin{equation}\label{eq:U-good-rev}
h_s(t)\ge I_{P,f}(s)-2\eta\qquad\text{for all }t\in U.
\end{equation}
Let
\begin{equation*}
\mathcal{J}_n(U):=\{\alpha\in nP\cap\Z_+^d:\ \alpha/n\in U\}.
\end{equation*}
Since $P$ is a convex body, $\mathrm{int}(P)$ is dense in $P$, hence $U\cap\mathrm{int}(P)\neq\varnothing$.
Fix $t_1\in U\cap\mathrm{int}(P)$. In particular, as $P\subset\R_+^d$, we have $(t_1)_j>0$ for all $j$.
Since $t_1\in\mathrm{int}(P)$, there exists $r>0$ such that $B(t_1,r)\subset P$.
Choose $\kappa>0$ so small that $\kappa<\min_{1\le j\le d}(t_1)_j$ and
\begin{equation*}
t_1+[-\kappa,\kappa]^d\subset B(t_1,r)\cap B(t_0,\rho)\subset U,
\end{equation*}
and set $Q_U:=t_1+[-\kappa,\kappa]^d\subset U$.
Then $nQ_U\cap\Z_+^d\subset \mathcal{J}_n(U)$ for all $n$.
Moreover, for all large $n$ one has $|nQ_U\cap\Z_+^d|\ge (2\kappa n-1)^d$, hence there exist $c_U>0$ and $n_1$ such that
\begin{equation}\label{eq:JnU-size-rev}
|\mathcal{J}_n(U)|\ge c_U n^d,\qquad n\ge n_1.
\end{equation}

By \eqref{eq:profile}, there exists $n_2=n_2(\eta)$ such that for all $n\ge n_2$ and all $\alpha\in nP\cap\Z_+^d$,
\begin{equation}\label{eq:coef-comp-rev}
\frac1n\log|c_{\alpha,n}|\ge \log f(\alpha/n)-\eta.
\end{equation}
Hence for $n\ge n_2$ and $\alpha\in\mathcal{J}_n(U)$, using \eqref{eq:U-good-rev} and $|z|^\alpha=e^{\langle s,\alpha\rangle}$ with $s=\Log(z)$, we obtain
\begin{equation*}
\frac1n\log\big(|c_{\alpha,n}||z|^\alpha\big)
\ge \big(\log f(\alpha/n)-\eta\big)+\Big\langle s,\frac{\alpha}{n}\Big\rangle
= h_s(\alpha/n)-\eta\ge I_{P,f}(s)-3\eta.
\end{equation*}
Multiplying by $n$ and exponentiating both sides, we find that for all such $\alpha$,
\begin{equation}\label{eq:large-coefs-rev}
|c_{\alpha,n}||z|^\alpha \ge \exp\big(n(I_{P,f}(s)-3\eta)\big).
\end{equation}

Define the normalized polynomial
\begin{equation*}
\widetilde{\mathbf{P}}_n(z):=e^{-n(I_{P,f}(s)-3\eta)}\mathbf{P}_n(z)
=\sum_{\alpha\in nP\cap\Z_+^d}\xi_\alpha a_{\alpha,n},
\qquad\text{where}\quad
a_{\alpha,n}:=c_{\alpha,n}z^\alpha\,e^{-n(I_{P,f}(s)-3\eta)}.
\end{equation*}
We decompose this as $\widetilde{\mathbf{P}}_n(z)=S_n+T_n$, where
\begin{equation*}
S_n:=\sum_{\alpha\in\mathcal{J}_n(U)}\xi_\alpha a_{\alpha,n}
\qquad\text{and}\qquad
T_n:=\sum_{\alpha\in (nP\cap\Z_+^d)\setminus\mathcal{J}_n(U)}\xi_\alpha a_{\alpha,n}.
\end{equation*}
Then $S_n$ and $T_n$ are independent, and by \eqref{eq:large-coefs-rev} we have $|a_{\alpha,n}|\ge 1$ for $\alpha\in\mathcal{J}_n(U)$.
Let $r_n:=e^{-\eta n}$. Using the concentration function \eqref{eq:conc-fn} and \eqref{eq:conc-basic-ineq} (and the independence of $S_n$ and $T_n$), we obtain
\begin{equation*}
\mathbb{P}\big(|\widetilde{\mathbf{P}}_n(z)|\le r_n\big)
\le \mathcal Q(\widetilde{\mathbf{P}}_n(z);r_n)
=\mathcal Q(S_n+T_n;r_n)\le \mathcal Q(S_n;r_n).
\end{equation*}
Equivalently,
\begin{equation*}
\mathbb{P}\big(|\mathbf{P}_n(z)|\le e^{n(I_{P,f}(s)-4\eta)}\big)
\le \mathcal Q(S_n;r_n).
\end{equation*}

Applying the Kolmogorov--Rogozin inequality (Theorem~\ref{thm:KR}) to the independent family $(\xi_\alpha a_{\alpha,n})_{\alpha\in\mathcal{J}_n(U)}$ yields
\begin{equation*}
\mathcal Q(S_n;r_n)\le
\frac{C_{\mathrm{KR}}}{\sqrt{\sum_{\alpha\in\mathcal{J}_n(U)}\bigl(1-\mathcal Q(\xi_\alpha a_{\alpha,n};r_n)\bigr)}}.
\end{equation*}
Since $|a_{\alpha,n}|\ge 1$ and $\mathcal Q(aX;r)=\mathcal Q(X;r/|a|)$,
\begin{equation*}
\mathcal Q(\xi_\alpha a_{\alpha,n};r_n)
=\mathcal Q\!\left(\xi_0;\frac{r_n}{|a_{\alpha,n}|}\right)
\le \mathcal Q(\xi_0;r_n).
\end{equation*}
Non-degeneracy of $\xi_0$ implies that there exists $r_0>0$ such that $\mathcal Q(\xi_0;r_0)<1$.
Since $r\mapsto \mathcal Q(\xi_0;r)$ is nondecreasing, for all large $n$ we have $r_n\le r_0$ and hence
$\mathcal Q(\xi_0;r_n)\le \mathcal Q(\xi_0;r_0)$.
Setting $q:=1-\mathcal Q(\xi_0;r_0)>0$, we get $1-\mathcal Q(\xi_0;r_n)\ge q$ for all large $n$, and therefore,
using \eqref{eq:JnU-size-rev},
\begin{equation}\label{eq:lower-prob-rev}
\mathbb{P}\big(|\mathbf{P}_n(z)|\le e^{n(I_{P,f}(s)-4\eta)}\big)
\le \frac{C'}{n^{d/2}},
\end{equation}
where $C'>0$ is a constant independent of $n$.
Because $d \ge 1$, the right-hand side goes to zero. Equivalently,
\begin{equation}\label{eq:lower-tail-rev}
\lim_{n \to \infty} \mathbb{P}\!\left(\frac1n\log|\mathbf{P}_n(z)|\le I_{P,f}(s)-4\eta\right) = 0.
\end{equation}

\smallskip\noindent\emph{Completion of the proof of \eqref{eq:pot-conv-prob}.}
Assume \eqref{eq:logtail}. Fix $\varepsilon>0$. By \eqref{eq:upper-tail-prob},
\begin{equation*}
\lim_{n \to \infty} \mathbb P\!\left(\frac1n\log|\mathbf{P}_n(z)|>\Phi_{P,f}(z)+3\varepsilon\right) = 0.
\end{equation*}
Moreover, by \eqref{eq:lower-tail-rev} with $4\eta = \varepsilon$,
\begin{equation*}
\lim_{n \to \infty} \mathbb P\!\left(\frac1n\log|\mathbf{P}_n(z)|<\Phi_{P,f}(z)-\varepsilon\right) = 0.
\end{equation*}
Therefore $\frac1n\log|\mathbf{P}_n(z)|\to\Phi_{P,f}(z)$ in probability, proving \eqref{eq:pot-conv-prob}.
\end{proof}

\medskip
We will also need an almost sure pointwise statement along the extracted sparse subsubsequence.

\begin{lem}\label{lem:pointwise-as-goodsubseq}
Assume \eqref{eq:profile} and \eqref{eq:logtail}. Let $(n_k)_{k\ge1}$ be any increasing sequence of integers and
let $(m_j)_{j\ge1}$ be a further subsequence provided by Lemma~\ref{lem:good-subsequence}, with the associated event
$\Omega_\xi$ of \eqref{eq:good-subseq-bc}. Then for every fixed $z\in(\mathbb{C}^*)^d$,
\begin{equation*}
\frac1{m_j}\log|\mathbf{P}_{m_j}(z,\omega)|\longrightarrow \Phi_{P,f}(z)\qquad\text{almost surely as }j\to\infty.
\end{equation*}
\end{lem}
\begin{proof}
Fix $z\in(\mathbb{C}^*)^d$ and an arbitrary $\eta>0$. Set $s:=\operatorname{Log}(z)\in\mathbb{R}^d$ and denote $u_n:=\frac1n\log|\mathbf{P}_n|$. 
Since the extracted subsequence satisfies $m_j\ge j^3$, we have $\sum_{j=1}^\infty m_j^{-d/2}<\infty$.

Now we first establish the lower bound. The quantitative estimate \eqref{eq:lower-prob-rev} implies that for all large $j$, we have
\begin{equation*}
\mathbb{P}\!\left(u_{m_j}(z)\le I_{P,f}(s)-4\eta\right)\le \frac{C'}{m_j^{d/2}}.
\end{equation*}
Hence $\sum_{j=1}^\infty \mathbb{P}(u_{m_j}(z)\le I_{P,f}(s)-4\eta)<\infty$, and the Borel--Cantelli lemma yields
\begin{equation}\label{eq:lb-liminf-mj}
\liminf_{j\to\infty}u_{m_j}(z,\omega)\ge I_{P,f}(s)-4\eta=\Phi_{P,f}(z)-4\eta
\qquad\text{almost surely.}
\end{equation}

For the corresponding upper bound, fix $\omega\in\Omega_\xi$. By \eqref{eq:good-subseq-bc}, for all large $j$ and all $\alpha\in m_jP\cap\mathbb{Z}_+^d$ we have $|\xi_\alpha(\omega)|\le e^{m_j/j}$. Moreover, by \eqref{eq:profile} there exists $j_0$ such that for $j\ge j_0$ and all $\alpha\in m_jP\cap\mathbb{Z}_+^d$,
\begin{equation*}
\frac1{m_j}\log|c_{\alpha,m_j}|\le \log f(\alpha/m_j)+\eta.
\end{equation*}
Thus, for $j\ge j_0$ and all $\alpha\in m_jP\cap\mathbb{Z}_+^d$,
\begin{equation*}
|\xi_\alpha(\omega)c_{\alpha,m_j}z^\alpha|
\le \exp\Big(m_j\big(\langle s,\alpha/m_j\rangle+\log f(\alpha/m_j)+\eta\big)\Big)\,e^{m_j/j}
\le \exp\big(m_j(I_{P,f}(s)+\eta)\big)\,e^{m_j/j},
\end{equation*}
and therefore via the triangle inequality,
\begin{equation*}
|\mathbf{P}_{m_j}(z,\omega)|
\le |m_jP\cap\mathbb{Z}_+^d|\,\exp\big(m_j(I_{P,f}(s)+\eta)\big)\,e^{m_j/j}.
\end{equation*}
Taking logarithms and dividing by $m_j$ gives, for all large $j$,
\begin{equation}\label{eq:ub-pointwise-mj}
u_{m_j}(z,\omega)\le \Phi_{P,f}(z)+\eta+\frac1j+\frac1{m_j}\log|m_jP\cap\mathbb{Z}_+^d|.
\end{equation}
Hence $\limsup_{j\to\infty}u_{m_j}(z,\omega)\le \Phi_{P,f}(z)+\eta$ for every $\omega\in\Omega_\xi$.

Since $\eta>0$ is arbitrary, we may take a countable sequence $\eta_k \downarrow 0$. Intersecting the probability-one event $\Omega_\xi$ with the corresponding almost sure events from the lower bound \eqref{eq:lb-liminf-mj}, we deduce that almost surely,
\begin{equation*}
\Phi_{P,f}(z) \le \liminf_{j\to\infty}u_{m_j}(z,\omega) \le \limsup_{j\to\infty}u_{m_j}(z,\omega) \le \Phi_{P,f}(z).
\end{equation*}
This shows $u_{m_j}(z,\omega)\to \Phi_{P,f}(z)$ almost surely, completing the proof.
\end{proof}
\medskip
Next we prove the following lemmas which will be used to upgrade pointwise convergence to $L^1_{\mathrm{loc}}$ convergence in probability.

\begin{lem}\label{lem:fubini-swap1}
Let $v_n:\C^d\times\Omega\to[-\infty,\infty)$ be jointly measurable and let $v:\C^d\to[-\infty,\infty)$ be measurable.
Assume that for Lebesgue-a.e.\ $z\in\C^d$,
\begin{equation*}
v_n(z,\omega)\to v(z)\qquad\text{almost surely in }\omega.
\end{equation*}
Then there exists an event $\Omega_0\subset\Omega$ with $\mathbb P(\Omega_0)=1$ such that for every $\omega\in\Omega_0$,
\begin{equation*}
v_n(z,\omega)\to v(z)\qquad\text{for Lebesgue-a.e.\ } z\in\C^d.
\end{equation*}
\end{lem}

\begin{proof}
Let
\begin{equation*}
\mathcal A:=\{(z,\omega)\in \C^d\times\Omega:\ v_n(z,\omega)\to v(z)\}.
\end{equation*}
Then $\mathcal A$ is measurable. For $z\in\C^d$ write $\mathcal A_z:=\{\omega:\ (z,\omega)\in\mathcal A\}$ and for $\omega\in\Omega$
write $\mathcal A_\omega:=\{z:\ (z,\omega)\in\mathcal A\}$. The assumption states that $\mathbb{P}(\mathcal A_z^c)=0$ for Lebesgue-a.e.\ $z\in\C^d$.
By Tonelli's theorem applied to the indicator function $\mathbf{1}_{\mathcal{A}^c}$,
\begin{equation*}
0=\int_{\C^d}\mathbb{P}(\mathcal A_z^c)\,d\mathrm{Leb}(z)
=\int_{\C^d}\int_\Omega \mathbf{1}_{\mathcal{A}^c}(z,\omega)\,d\mathbb{P}(\omega)\,d\mathrm{Leb}(z)
=\int_\Omega \mathrm{Leb}(\mathcal A_\omega^c)\,d\mathbb{P}(\omega).
\end{equation*}
Hence $\mathrm{Leb}(\mathcal A_\omega^c)=0$ for $\mathbb{P}$-a.e.\ $\omega$.
Set $\Omega_0:=\{\omega\in\Omega:\ \mathrm{Leb}(\mathcal A_\omega^c)=0\}$. Then $\mathbb{P}(\Omega_0)=1$, and for every $\omega\in\Omega_0$ we have $\mathrm{Leb}(\mathcal A_\omega^c)=0$, which means $v_n(z,\omega)\to v(z)$ for Lebesgue-a.e.\ $z\in\C^d$.
\end{proof}

\begin{lem}[Eventual non-vanishing]\label{lem:eventual-nonzero}
Let $(\xi_\alpha)_{\alpha\in\Z_+^d}$ be i.i.d.\ complex-valued, non-degenerate random variables.
Then almost surely $\mathbf{P}_n(\cdot,\omega)\not\equiv0$ for all sufficiently large $n$.
\end{lem}

\begin{proof}
Since $f>0$ on $P$, the profile condition \eqref{eq:profile} forces $c_{\alpha,n}\neq0$ for every $\alpha\in nP\cap\Z_+^d$ once $n$ is sufficiently large. Indeed, if $c_{\alpha,n}=0$, then $\frac1n\log|c_{\alpha,n}|=-\infty$, which contradicts the uniform approximation of $\log f$ on the compact set $P$.
For such $n$, because the monomials are linearly independent, we have
\begin{equation*}
\mathbf{P}_n\equiv0 \iff \xi_\alpha=0\ \text{for all }\alpha\in nP\cap\Z_+^d.
\end{equation*}
Set $p_0:=\mathbb{P}(\xi_0=0)\in[0,1)$. By the independence of the random variables $(\xi_\alpha)$,
\begin{equation*}
\mathbb{P}(\mathbf{P}_n\equiv0)=p_0^{\,|nP\cap\Z_+^d|}.
\end{equation*}
By Lemma~\ref{lem:lattice-growth}, there exist $c>0$ and $n_0$ such that $|nP\cap\Z_+^d|\ge c n^d$ for all $n\ge n_0$. Hence
\begin{equation*}
\sum_{n=1}^\infty \mathbb{P}(\mathbf{P}_n\equiv0)\le \sum_{n<n_0}1+\sum_{n\ge n_0} p_0^{c n^d}<\infty.
\end{equation*}
Borel--Cantelli lemma implies that the event $\{\mathbf{P}_n\equiv0\}$ occurs only finitely many times almost surely.
\end{proof}

We now utilize the preceding results to prove Theorem~\ref{thm:potentials} and Corollary~\ref{cor:currents}.

\begin{proof}[Proof of Theorem~\ref{thm:potentials}]
Assume \eqref{eq:logtail}. We use the convention from the Introduction, namely we set $u_n:=0$ and $[Z_{\mathbf P_n}]:=0$ on the event $\{\mathbf P_n\equiv 0\}$; this does not affect any limit statement by Lemma~\ref{lem:eventual-nonzero}.

We emphasize that Theorem~\ref{thm:KZ-P-Theorem41} provides only \emph{pointwise} convergence in probability. 
To upgrade this to convergence in probability in the separable metric space $L^1_{\mathrm{loc}}(\mathbb{C}^d)$, we invoke the standard subsequence principle. Specifically, it suffices to show that every increasing sequence $(n_k)$ admits a further subsequence $(m_j)$ along which
\begin{equation}\label{eq:L1-subseq-property}
u_{m_j}\longrightarrow \Phi_{P,f}\qquad\text{almost surely in }L^1_{\mathrm{loc}}(\mathbb{C}^d).
\end{equation}
Establishing this requires combining the almost sure pointwise convergence along a sparse extracted subsubsequence (Lemma~\ref{lem:pointwise-as-goodsubseq}) with local uniform upper bounds on compact sets.

Fix an increasing sequence $(n_k)$. By Lemma~\ref{lem:good-subsequence} there exist a subsequence $(m_j)$ with $m_j\ge j^3$ and an event $\Omega_\xi\subset\Omega$ with $\mathbb{P}(\Omega_\xi)=1$ such that
\begin{equation}\label{eq:good-subseq-used}
\lim_{j\to\infty}\frac1{m_j}\max_{\alpha\in m_jP\cap\mathbb{Z}_+^d}\log\big(1+|\xi_\alpha|\big)=0
\qquad\text{on }\Omega_\xi.
\end{equation}

\smallskip\noindent\emph{Step 1: uniform upper bounds on compacts.}
Let $K\subset\mathbb{C}^d$ be compact and fix $\eta>0$. Set
\begin{equation*}
R_{K,i}:=\max\Big\{1,\sup_{z\in K}|z_i|\Big\},\qquad S_K:=(\log R_{K,1},\dots,\log R_{K,d})\in\mathbb{R}_+^d,
\qquad B_K:=I_{P,f}(S_K).
\end{equation*}
By the profile condition \eqref{eq:profile} there exists $N_\eta$ such that for all $n\ge N_\eta$ and all $\alpha\in nP\cap\mathbb{Z}_+^d$,
\begin{equation}\label{eq:profile-upper-used}
|c_{\alpha,n}|\le \exp\big(n(\log f(\alpha/n)+\eta)\big).
\end{equation}
Fix $\omega\in\Omega_\xi$. By \eqref{eq:good-subseq-used} there exists $j_0=j_0(\omega,\eta)$ such that for all $j\ge j_0$ and all $\alpha\in m_jP\cap\mathbb{Z}_+^d$,
\begin{equation}\label{eq:xi-upper-used}
|\xi_\alpha(\omega)|\le \exp(m_j\eta).
\end{equation}
Using \eqref{eq:profile-upper-used}, \eqref{eq:xi-upper-used} and $|z^\alpha|\le \prod_i R_{K,i}^{\alpha_i}= \exp(\langle \alpha,S_K\rangle)$ for $z\in K$, we obtain for $m_j \ge \max\{N_\eta, m_{j_0}\}$:
\begin{equation*}
\sup_{z\in K}|\mathbf P_{m_j}(z,\omega)|
\le |m_jP\cap\mathbb{Z}_+^d|\exp\big(m_j(B_K+2\eta)\big).
\end{equation*}
Hence,
\begin{equation}\label{eq:ub-K-mj}
\sup_{z\in K}u_{m_j}(z,\omega)
\le B_K+2\eta+\frac1{m_j}\log|m_jP\cap\mathbb{Z}_+^d|.
\end{equation}
Because the first finitely many terms are upper semicontinuous and bounded above on the compact set $K$, by \eqref{eq:good-subseq-used} we conclude that for every $\omega\in\Omega_\xi$, the entire family $(u_{m_j}(\cdot,\omega))_{j\ge1}$ is locally uniformly bounded above on $\mathbb{C}^d$.

\smallskip\noindent\emph{Step 2: almost everywhere convergence along $(m_j)$.}
For every fixed $z\in(\mathbb{C}^*)^d$, Lemma~\ref{lem:pointwise-as-goodsubseq} yields
\begin{equation*}
\mathbb{P}\big(\{\omega:\ u_{m_j}(z,\omega)\to \Phi_{P,f}(z)\}\big)=1.
\end{equation*}
Moreover, $(z,\omega)\mapsto u_{m_j}(z,\omega)$ is jointly measurable on $\mathbb{C}^d\times\Omega$ for each $j$ (being continuous in $z$ and measurable in $\omega$).
Applying Lemma~\ref{lem:fubini-swap1} we obtain an event $\Omega_0\subset\Omega$ with $\mathbb{P}(\Omega_0)=1$
such that for every $\omega\in\Omega_0$,
\begin{equation*}
u_{m_j}(z,\omega)\to \Phi_{P,f}(z)\quad\text{for Lebesgue-a.e. }z\in(\mathbb{C}^*)^d
\quad(\text{hence for a.e. }z\in\mathbb{C}^d).
\end{equation*}

\smallskip\noindent\emph{Step 3: almost sure $L^1_{\mathrm{loc}}$-convergence along $(m_j)$.}
Fix $\omega\in\Omega_\xi\cap\Omega_0$. For each $j$, the function $u_{m_j}(\cdot,\omega)$ is plurisubharmonic on $\mathbb{C}^d$ (using our convention $u_{m_j} \equiv 0$ on the exceptional event $\{\mathbf{P}_{m_j} \equiv 0\}$).

Moreover, Step~1 yields a local uniform upper bound on compacts for the family $(u_{m_j}(\cdot,\omega))_{j\ge1}$, and Step~2 gives
$u_{m_j}(\cdot,\omega)\to \Phi_{P,f}$ for Lebesgue-a.e.\ $z\in\mathbb{C}^d$. 
Recall that $\Phi_{P,f}\in\operatorname{PSH}(\mathbb{C}^d)$ and $\Phi_{P,f}\not\equiv-\infty$ (as it is finite on $(\mathbb{C}^*)^d$). Applying Corollary~\ref{cor:psh-ae-to-L1} we obtain
\begin{equation*}
u_{m_j}(\cdot,\omega)\longrightarrow \Phi_{P,f}\qquad\text{in }L^1_{\mathrm{loc}}(\mathbb{C}^d).
\end{equation*}
This proves \eqref{eq:L1-subseq-property}, and hence assertion \eqref{eq:L1prob} follows.
\end{proof}

\begin{proof}[Proof of Corollary~\ref{cor:currents}]
For each $n$, on the event $\{\mathbf{P}_n\not\equiv0\}$ we have
$u_n=\frac1n\log|\mathbf{P}_n|\in\operatorname{PSH}(\mathbb{C}^d)$, hence by the Poincar\'e--Lelong formula,
\begin{equation*}
dd^c u_n=\frac1n\,dd^c\log|\mathbf{P}_n|=\frac1n[Z_{\mathbf{P}_n}].
\end{equation*}
On the exceptional event $\{\mathbf{P}_n\equiv0\}$, both sides are $0$ by our convention. Thus the identity $dd^c u_n=\frac1n[Z_{\mathbf{P}_n}]$
holds everywhere. By Lemma~\ref{lem:eventual-nonzero}, $\{\mathbf{P}_n\equiv0\}$ occurs only finitely many times almost surely, so it does not affect any asymptotic limits.

\smallskip\noindent\emph{(i) Almost sure convergence along further subsequences.}
Let $(n_k)$ be an arbitrary subsequence. By Theorem~\ref{thm:potentials} and the subsequence principle for the metric space $(L^1_{\mathrm{loc}}(\mathbb{C}^d),d_{L^1_{\mathrm{loc}}})$, there exists a further subsequence $(n_{k_j})$ such that $u_{n_{k_j}}\to \Phi_{P,f}$ almost surely in $L^1_{\mathrm{loc}}(\mathbb{C}^d)$. The continuity of $dd^c:L^1_{\mathrm{loc}}(\mathbb{C}^d)\to \mathcal{D}'^{\,1,1}(\mathbb{C}^d)$ then directly yields
\begin{equation*}
\frac1{n_{k_j}}[Z_{\mathbf{P}_{n_{k_j}}}]
= dd^c u_{n_{k_j}}
\longrightarrow dd^c\Phi_{P,f}
\qquad \text{almost surely in the weak sense of currents.}
\end{equation*}

\smallskip\noindent\emph{(ii) Convergence in probability in the weak topology of currents.}
Let $\mathcal{O}\subset\mathcal{D}'^{\,1,1}(\mathbb{C}^d)$ be an open neighborhood of $dd^c\Phi_{P,f}$ for the weak topology of currents.
Because the operator $dd^c:L^1_{\mathrm{loc}}(\mathbb{C}^d)\to \mathcal{D}'^{\,1,1}(\mathbb{C}^d)$
is continuous with respect to the $L^1_{\mathrm{loc}}$ topology and the weak topology of currents, the preimage
\begin{equation*}
U:=(dd^c)^{-1}(\mathcal{O})
=\{v\in L^1_{\mathrm{loc}}(\mathbb{C}^d):\ dd^c v\in\mathcal{O}\}
\end{equation*}
is open in $L^1_{\mathrm{loc}}(\mathbb{C}^d)$ and contains $\Phi_{P,f}$.
By Theorem~\ref{thm:potentials}, $u_n\to \Phi_{P,f}$ in probability 
$L^1_{\mathrm{loc}}(\mathbb{C}^d)$, hence
\begin{equation*}
\mathbb{P}\!\left(\frac1n[Z_{\mathbf{P}_n}]\in \mathcal{O}\right)
=\mathbb{P}\big(dd^c u_n\in \mathcal{O}\big)
=\mathbb{P}(u_n\in U)\xrightarrow[n\to\infty]{}1.
\end{equation*}
\end{proof}

\section{Almost sure convergence}\label{sec:conv-as}

In this section we prove Theorem~\ref{thm:main}. Under the logarithmic moment condition \eqref{eq:logmoment}, we first obtain
almost sure multi-index subexponential bounds on the coefficients. We then combine Borel--Cantelli arguments with standard
plurisubharmonic compactness/uniqueness principles to deduce almost sure weak convergence of the normalized zero currents.

\subsection{Coefficient control under the moment assumption}\label{subsec:as-coeff}

We begin with an almost sure multi-index subexponential bound of random coefficients.

\begin{lem}\label{lem:multiindex-subexp}
Assume \eqref{eq:logmoment}. Then for every $\delta>0$, almost surely there exists a finite random constant
$C_\delta=C_\delta(\omega)\ge 1$ such that, for all $\alpha\in\Z_+^d$,
\begin{equation}\label{eq:xi-subexp}
|\xi_\alpha(\omega)|\le C_\delta(\omega)\,e^{\delta|\alpha|},
\end{equation}
where $|\alpha|:=\alpha_1+\cdots+\alpha_d$.
\end{lem}

\begin{proof}
Fix $\delta>0$ and set $Y_\alpha:=\log(1+|\xi_\alpha|)\ge 0$. For each $m\in\mathbb{N}$ define 
\begin{equation*}
\mathcal I_m:=\{\alpha\in\Z_+^d:\ |\alpha|=m\},
\end{equation*}
which has cardinality $N_m:=|\mathcal I_m|=\binom{m+d-1}{d-1}\le C_d\,m^{d-1}.$
Consider the events
\begin{equation*}
A_m:=\Big\{\max_{\alpha\in\mathcal{I}_m} Y_\alpha > \delta m\Big\}
=\bigcup_{\alpha\in\mathcal I_m}\{Y_\alpha>\delta m\}.
\end{equation*}
By the union bound and identical distribution,
\begin{equation*}
\mathbb{P}(A_m)\le \sum_{\alpha\in\mathcal I_m}\mathbb{P}(Y_\alpha>\delta m)
= N_m\,\mathbb{P}(Y_0>\delta m)
\le C_d\,m^{d-1}\mathbb{P}(Y_0>\delta m).
\end{equation*}
Thus it suffices to show $\sum_{m\ge1} m^{d-1}\mathbb{P}(Y_0>\delta m)<\infty$.

Using the tail-integral identity and $\E[Y_0^d]<\infty$, we have
\begin{equation*}
\E[Y_0^d]=d\int_0^\infty t^{d-1}\mathbb{P}(Y_0>t)\,dt<\infty.
\end{equation*}
Since $t\mapsto\mathbb{P}(Y_0>t)$ is decreasing, for $m\ge2$,
\begin{equation*}
\int_{\delta(m-1)}^{\delta m} t^{d-1}\mathbb{P}(Y_0>t)\,dt
\ge \delta\,(\delta(m-1))^{d-1}\mathbb{P}(Y_0>\delta m).
\end{equation*}
Hence for $m\ge2$,
\begin{equation*}
m^{d-1}\mathbb{P}(Y_0>\delta m)
\le \frac{2^{d-1}}{\delta^d}\int_{\delta(m-1)}^{\delta m} t^{d-1}\mathbb{P}(Y_0>t)\,dt,
\end{equation*}
using $m^{d-1}\le 2^{d-1}(m-1)^{d-1}$. Summing over $m\ge2$ gives
$\sum_{m\ge1} m^{d-1}\mathbb{P}(Y_0>\delta m)<\infty$, hence $\sum_{m\ge1}\mathbb{P}(A_m)<\infty$.

By Borel--Cantelli, almost surely there exists $m_0(\omega)$ such that for all $m\ge m_0(\omega)$, $A_m$ does not occur; i.e.
$Y_\alpha(\omega)\le \delta|\alpha|$ for all $|\alpha|\ge m_0(\omega)$. Exponentiating gives $|\xi_\alpha(\omega)|\le e^{\delta|\alpha|}$ for all $|\alpha|\ge m_0(\omega)$.
Set
\begin{equation*}
C_\delta(\omega):=\max\Big\{1,\ \max_{|\alpha|<m_0(\omega)}|\xi_\alpha(\omega)|e^{-\delta|\alpha|}\Big\}.
\end{equation*}
Then $C_\delta(\omega)$ is a.s.\ finite and \eqref{eq:xi-subexp} holds for all $\alpha\in\Z_+^d$.
\end{proof}

We will use a single probability-one event on which \eqref{eq:xi-subexp} holds simultaneously for all $\delta>0$.

\begin{rem}\label{rem:Omega-xi}
Applying Lemma~\ref{lem:multiindex-subexp} along the countable sequence $\delta=1/j$ for $j\in\mathbb{N}$, we may define the probability-one event $\Omega^\xi:=\bigcap_{j=1}^\infty \Omega(1/j)$. For any $\omega\in\Omega^\xi$ and any $\delta>0$, choosing $j$ such that $1/j\le\delta$ ensures that \eqref{eq:xi-subexp} holds everywhere on $\Omega^\xi$ with the finite constant $C_\delta(\omega):=C_{1/j}(\omega)$.
\end{rem}

\begin{lem}\label{lem:upper-bounds-compacts}
Assume \eqref{eq:logmoment} and \eqref{eq:profile}. Then almost surely the family $u_n(\cdot,\omega):=\frac1n\log|\mathbf{P}_n(\cdot,\omega)|$ is locally uniformly bounded above on $\mathbb{C}^d$; that is, for every compact set $K\Subset\mathbb{C}^d$,
\begin{equation*}
\sup_{n\ge1}\sup_{z\in K} u_n(z,\omega) < \infty \qquad \text{almost surely.}
\end{equation*}
\end{lem}

\begin{proof}
Fix a compact set $K\Subset\mathbb{C}^d$ and $\varepsilon>0$. Define
\begin{equation*}
R_{K,i}:=\max\Big\{1,\sup_{z\in K}|z_i|\Big\}\quad (i=1,\dots,d),\qquad
S_K:=(\log R_{K,1},\dots,\log R_{K,d})\in\mathbb{R}_+^d,
\end{equation*}
and set $B_K:=\sup_{t\in P}\big(\langle S_K,t\rangle+\log f(t)\big)=I_{P,f}(S_K)<\infty$. 
Let $M_P:=\sup_{t\in P}|t|_1<\infty$ and choose $\delta>0$ so that $\delta M_P\le\varepsilon$. Work on the probability-one event $\Omega^\xi$ from Remark~\ref{rem:Omega-xi}. Then for each $\omega\in\Omega^\xi$, there exists $C_\delta(\omega)<\infty$ such that $|\xi_\alpha(\omega)|\le C_\delta(\omega)e^{\delta|\alpha|}$ for all $\alpha$. If $\alpha\in nP\cap\mathbb{Z}_+^d$, then $|\alpha|\le nM_P$, hence for all $\alpha\in nP\cap\mathbb{Z}_+^d,$
\begin{equation}\label{eq:xi-K}
|\xi_\alpha(\omega)|\le C_\delta(\omega)e^{\varepsilon n}.
\end{equation}

By the profile condition \eqref{eq:profile}, there exists $n_0=n_0(\varepsilon)$ such that for all $n\ge n_0$ and $\alpha\in nP\cap\mathbb{Z}_+^d$,
\begin{equation*}
\frac1n\log|c_{\alpha,n}|\le \log f(\alpha/n)+\varepsilon.
\end{equation*}
For $z\in K$, we have $|z^\alpha|\le \exp(\langle S_K,\alpha\rangle)$, and since  $\alpha/n\in P$,
\begin{equation*}
|c_{\alpha,n}z^\alpha|
\le \exp\Big(n\big(\log f(\alpha/n)+\varepsilon+\langle S_K,\alpha/n\rangle\big)\Big)
\le \exp\big(n(B_K+\varepsilon)\big).
\end{equation*}
Combining this with \eqref{eq:xi-K} and summing over $\alpha$ yields, for $\omega\in\Omega^\xi$, $n\ge n_0$, and $z\in K$,
\begin{equation*}
|\mathbf{P}_n(z,\omega)|
\le C_\delta(\omega)\,|nP\cap\mathbb{Z}_+^d|\,\exp\big(n(B_K+2\varepsilon)\big).
\end{equation*}
By Lemma~\ref{lem:lattice-growth}, $\frac1n\log|nP\cap\mathbb{Z}_+^d|\to0$, so there exists $n_1=n_1(\varepsilon)$ such that $\frac1n\log|nP\cap\Z_+^d|\le\varepsilon$ for all $n\ge n_1$. Therefore, for $n\ge N:=\max\{n_0,n_1\}$,
\begin{equation*}
\sup_{z\in K} u_n(z,\omega)\le B_K+3\varepsilon+\frac1n\log C_\delta(\omega).
\end{equation*}
Absorbing the finitely many indices $1\le n<N$ (each $\sup_{z\in K}u_n(z,\omega)$ is finite) yields the desired bound on $K$.
\end{proof}

\subsection{Pointwise limits and proof of Theorem~\ref{thm:main}}\label{subsec:ptwise-main-proof}

When $d>2$ the lower-tail bound in \eqref{eq:lower-prob-rev} is summable, so we obtain pointwise almost sure convergence along the full sequence.

\begin{thm}\label{thm:KZ-P-as}
Assume \eqref{eq:profile} and \eqref{eq:logmoment}, and suppose that $d>2$. Then for every $z\in(\C^*)^d$,
\begin{equation}\label{eq:pot-conv-as}
\frac1n\log|\mathbf{P}_n(z)|\longrightarrow \Phi_{P,f}(z)\qquad\text{almost surely as }n\to\infty.
\end{equation}
\end{thm}

\begin{proof}
Fix $z\in(\C^*)^d$ and $\varepsilon>0$ and work on the event $\Omega^\xi$ from Remark~\ref{rem:Omega-xi}.
Set $M_P:=\sup_{t\in P}|t|_1$ and choose $\delta:=\varepsilon/(2M_P)$. For $\omega\in\Omega^\xi$ there exists $C_\delta(\omega)$ such that
$|\xi_\alpha(\omega)|\le C_\delta(\omega)e^{\delta|\alpha|}$ for all $\alpha$. Hence for every $n\ge1$,
\begin{equation*}
\max_{\alpha\in nP\cap\Z_+^d}\log\big(1+|\xi_\alpha(\omega)|\big)
\le \log(2C_\delta(\omega))+\delta\,\max_{\alpha\in nP\cap\Z_+^d}|\alpha|
\le \log(2C_\delta(\omega))+\delta n M_P.
\end{equation*}
Since $\delta M_P=\varepsilon/2$, there exists $n_0=n_0(\omega,\varepsilon)$ such that for all $n\ge n_0$,
\begin{equation*}
\max_{\alpha\in nP\cap\Z_+^d}\log\big(1+|\xi_\alpha(\omega)|\big)\le \varepsilon n.
\end{equation*}
Repeating the deterministic estimate from the upper bound in the proof of Theorem~\ref{thm:KZ-P-Theorem41}, we obtain
\begin{equation}\label{eq:upper-limsup-rev}
\limsup_{n\to\infty}\frac1n\log|\mathbf{P}_n(z)|\le \Phi_{P,f}(z)+3\varepsilon
\qquad\text{on }\Omega^\xi.
\end{equation}

On the other hand, \eqref{eq:lower-prob-rev} yields
\begin{equation*}
\mathbb P\!\left(\frac1n\log|\mathbf{P}_n(z)|\le \Phi_{P,f}(z)-4\varepsilon\right)\le C'\,n^{-d/2}.
\end{equation*}
Since $d>2$, $\sum_{n\ge1}n^{-d/2}<\infty$, so the Borel--Cantelli lemma implies
\begin{equation*}
\liminf_{n\to\infty}\frac1n\log|\mathbf{P}_n(z)|\ge \Phi_{P,f}(z)-4\varepsilon
\qquad\text{almost surely}.
\end{equation*}
Intersecting with $\Omega^\xi$ and combining with \eqref{eq:upper-limsup-rev} gives
\begin{equation*}
\Phi_{P,f}(z)-4\varepsilon\le \liminf_{n\to\infty}\frac1n\log|\mathbf{P}_n(z)|
\le \limsup_{n\to\infty}\frac1n\log|\mathbf{P}_n(z)|\le \Phi_{P,f}(z)+3\varepsilon
\quad\text{a.s.}
\end{equation*}
Letting $\varepsilon\downarrow0$ along a countable sequence proves \eqref{eq:pot-conv-as}.
\end{proof}

\begin{rem}\label{rem:subseq-as}
When $d\le 2$, the lower-tail bound \eqref{eq:lower-prob-rev} is of order $O(n^{-d/2})$, which is not summable. Thus, the Borel--Cantelli lemma does not yield almost sure convergence along the full sequence.

This limitation is standard when applying the Kolmogorov--Rogozin inequality to polynomials with $O(n^d)$ terms. As demonstrated by Bloom and Dauvergne \cite{BD19}, one can overcome this barrier for $d \le 2$ by using stronger small-ball estimates derived from Nguyen and Vu \cite{NV11}. However, applying these  theorems requires the deterministic coefficients to satisfy certain metric covering conditions. 

While such conditions can be verified for specific models (such as the Kac ensemble), establishing them for an arbitrary polytope $P$ and an arbitrary profile function $f$ is highly non-trivial. To keep our framework valid for general $P$ and $f$ without imposing restrictive structural assumptions, we do not pursue this direction. Instead, part~\textup{(i)} of Theorem~\ref{thm:main} simply restricts to a sparse subsequence $(\ell_k)_{k\ge1}$ with $\sum_{k=1}^\infty \ell_k^{-d/2}<\infty$ to obtain almost sure convergence in $L^1_{\mathrm{loc}}(\mathbb{C}^d)$.
\end{rem}

\begin{proof}[Proof of Theorem~\ref{thm:main}]
We maintain the convention from the Introduction on $\{\mathbf P_n\equiv0\}$ (cf.\ Lemma~\ref{lem:eventual-nonzero}).
In particular, for almost every $\omega$ there exists $n_0(\omega)$ such that $\mathbf{P}_n(\cdot,\omega)\not\equiv0$ for all $n\ge n_0(\omega)$,
and hence $u_n(\cdot,\omega)=\frac1n\log|\mathbf{P}_n(\cdot,\omega)|\in\PSH(\C^d)$ for all sufficiently large $n$.

\smallskip\noindent\emph{(i) Convergence along deterministic sparse subsequences.}
Let $(\ell_k)_{k\ge1}$ be increasing with $\sum_{k=1}^\infty \ell_k^{-d/2}<\infty$.
Fix $z\in(\C^*)^d$ and $\eta>0$. Arguing as in the previous lemma for the lower bound, the Borel--Cantelli lemma yields
\begin{equation}\label{eq:liminf-sparse}
\liminf_{k\to\infty}\frac1{\ell_k}\log|\mathbf{P}_{\ell_k}(z)|\ge \Phi_{P,f}(z)-4\eta
\qquad\text{almost surely}.
\end{equation}
On the other hand, restricting to the probability-one event $\Omega^\xi$, the upper bound argument (as in \eqref{eq:upper-limsup-rev}) gives that
\begin{equation*}
\limsup_{k\to\infty}\frac1{\ell_k}\log|\mathbf{P}_{\ell_k}(z,\omega)|\le \Phi_{P,f}(z)
\qquad\text{for every }\omega\in\Omega^\xi.
\end{equation*}
Letting $\eta\downarrow0$ along a countable sequence yields
\begin{equation*}
u_{\ell_k}(z,\omega)=\frac1{\ell_k}\log|\mathbf{P}_{\ell_k}(z,\omega)|\longrightarrow \Phi_{P,f}(z)
\qquad\text{almost surely as }k\to\infty.
\end{equation*}

Since $(z,\omega)\mapsto u_{\ell_k}(z,\omega)$ is jointly measurable on $\C^d\times\Omega$, applying Lemma~\ref{lem:fubini-swap1} provides an event $\Omega_0\subset\Omega$ with $\mathbb{P}(\Omega_0)=1$ such that for every $\omega\in\Omega_0$,
\begin{equation*}
u_{\ell_k}(z,\omega)\to \Phi_{P,f}(z)\qquad\text{for Lebesgue-a.e.\ }z\in\C^d.
\end{equation*}

Fix $\omega\in\Omega^\xi\cap\Omega_0$ such that $\mathbf{P}_n(\cdot,\omega)\not\equiv0$ for all sufficiently large $n$.
Then $u_{\ell_k}(\cdot,\omega)\in\PSH(\C^d)$ for all sufficiently large $k$. By Lemma~\ref{lem:upper-bounds-compacts}, the family $(u_{\ell_k}(\cdot,\omega))_{k\ge1}$ is locally uniformly bounded above on $\C^d$. Furthermore, recall that $\Phi_{P,f}\not\equiv-\infty$ (as it is finite on $(\C^*)^d$). Since $u_{\ell_k}(\cdot,\omega)\to\Phi_{P,f}$ almost everywhere, Corollary~\ref{cor:psh-ae-to-L1} yields
\begin{equation*}
u_{\ell_k}(\cdot,\omega)\longrightarrow\Phi_{P,f}\qquad\text{in }L^1_{\mathrm{loc}}(\C^d).
\end{equation*}
By the continuity of $dd^c$, this immediately gives $\frac{1}{\ell_k}[Z_{\mathbf{P}_{\ell_k}}]\longrightarrow dd^c\Phi_{P,f}$ almost surely in the weak sense of currents, proving \textup{(i)}.

\smallskip\noindent\emph{(ii) Convergence along the full sequence when $d>2$.}
Assume now that $d>2$. By Theorem~\ref{thm:KZ-P-as}, for every $z\in(\C^*)^d$, we have $u_n(z,\omega)\to \Phi_{P,f}$ almost surely.
Because this provides the exact same almost sure pointwise convergence as in part \textup{(i)} but for the full sequence $(u_n)$, the remainder of the proof is identical. Arguing exactly as above via Lemma~\ref{lem:fubini-swap1}, Lemma~\ref{lem:upper-bounds-compacts}, and Corollary~\ref{cor:psh-ae-to-L1}, we conclude that $u_n(\cdot,\omega)\longrightarrow\Phi_{P,f}$ in  $L^1_{\mathrm{loc}}(\C^d)$
almost surely. The continuity of $dd^c$ then immediately yields $\frac1n[Z_{\mathbf{P}_n}]\longrightarrow dd^c\Phi_{P,f}$ almost surely in the weak sense of currents, proving \textup{(ii)}.
\end{proof}

\section{Torus-invariant random orthogonal polynomials}\label{sec:torus-orthopoly}

In this section we show that random orthogonal polynomials associated with certain torus-invariant weight functions are a special case of our random $P$-polynomials, with $P$ given by the standard simplex. Under the present hypotheses, this recovers the conclusion of a theorem of Bayraktar \cite{Bay19}.

\subsection{The random orthogonal polynomial ensemble}

Let $d\ge 1$ and let $Q:\C^d\to[0,\infty)$ be a $\mathscr{C}^2$ weight invariant under the real torus action, i.e., $Q(z_1,\dots,z_d)=Q(|z_1|,\dots,|z_d|)$ for $z\in\C^d$.
Assume that there exist $\varepsilon>0$ and $R>0$ such that
\begin{equation}\label{eq:Q-growth-orthopoly}
Q(z)\ge (1+\varepsilon)\log\|z\|\qquad\text{for }\|z\|\ge R,
\end{equation}
where $\|\cdot\|$ is the Euclidean norm. By enlarging $R$ if necessary, we may assume $R\ge 1$ throughout. 

For $n\in\mathbb{N}$, we consider the weighted inner product
\begin{equation*}
\langle p,q\rangle_n:=\int_{\C^d} p(z)\,\overline{q(z)}\,e^{-2nQ(z)}\,dV_d(z),
\end{equation*}
and its induced norm $\|\cdot\|_n$, where $dV_d$ denotes Lebesgue measure on $\C^d$. Let us denote by  $\mathcal P_n:=\mathrm{span}\{z^\alpha:\alpha\in\Z_+^d,\ |\alpha|\le n\}$ the space of polynomials  of total degree at most $n$. The growth condition \eqref{eq:Q-growth-orthopoly} guarantees that the squared norms
$\|z^\alpha\|^{2}_n=\int_{\C^d}|z^\alpha|^2 e^{-2nQ}\,dV_d$ are finite for all $|\alpha|\le n$ once $n$ is large enough:
indeed, for $\|z\|\ge R$ and $|\alpha|\le n$,
\begin{equation*}
|z^\alpha|^2 e^{-2nQ(z)}
\le \|z\|^{2|\alpha|}\,\|z\|^{-2n(1+\varepsilon)}
\le \|z\|^{-2n\varepsilon},
\end{equation*}
so the tail is integrable whenever $2n\varepsilon>2d$. Thus, fixing $n_0:=\lfloor d/\varepsilon\rfloor+1$, the normalization constants below are well-defined for all $n\ge n_0$.

A fundamental feature of the torus invariance of $Q$ is that distinct monomials are mutually orthogonal in $L^2(\C^d, e^{-2nQ})$. Therefore, an orthonormal basis of $\mathcal P_n$ is obtained simply by normalizing the monomials, $e_{\alpha,n}(z):=c_{\alpha,n}\,z^\alpha$ for $|\alpha|\le n$ and $n\ge n_0$, where
\begin{equation}\label{eq:c-alpha-n-orthopoly}
c_{\alpha,n}:=\Big(\int_{\C^d}|z^\alpha|^2 e^{-2nQ(z)}\,dV_d(z)\Big)^{-1/2}.
\end{equation}
Consequently, the associated random orthogonal polynomial ensemble takes the form
\begin{equation}\label{eq:random-orthopoly-ensemble}
G_n(z)=\sum_{|\alpha|\le n}\xi_\alpha\,c_{\alpha,n}\,z^\alpha,
\end{equation}
with i.i.d.\ complex coefficients $(\xi_\alpha)$.
To relate \eqref{eq:random-orthopoly-ensemble} to our general model, it remains to identify the asymptotic exponential profile of the array $(c_{\alpha,n})$. The relevant probabilistic assumptions on $(\xi_\alpha)$ (depending on whether one applies the convergence-in-probability or almost sure result) will be imposed when invoking the corresponding theorem.

\subsection{Asymptotic exponential profile and the weighted extremal function}\label{subsec:orthonormal-profile}

We keep the notation and standing assumptions from the previous subsection. In particular, the orthonormal basis is given by normalized monomials $e_{\alpha,n}(z)=c_{\alpha,n}z^\alpha$ and the random ensemble has the form \eqref{eq:random-orthopoly-ensemble}. We now identify the asymptotic exponential profile of the array $(c_{\alpha,n})$.

For $n\ge n_0$ and $\alpha\in\Z_+^d$ with $|\alpha|\le n$, define 
\begin{equation*}
I_{\alpha,n}:=\int_{\C^d}|z^\alpha|^2 e^{-2nQ(z)}\,dV_d \qquad \text{and} \qquad c_{\alpha,n}:=I_{\alpha,n}^{-1/2}.
\end{equation*}

To evaluate this integral, we pass to logarithmic coordinates. Write $z_j=r_je^{i\theta_j}$ with $r_j>0$ and $\theta_j\in[0,2\pi)$. Then $dV_d(z)=\prod_{j=1}^d r_j\,dr_j\,d\theta_j$. Since $Q$ is torus-invariant, $Q(z)=Q(r)$. Therefore,
\begin{equation*}
I_{\alpha,n}
=\int_{[0,2\pi)^d}\int_{(0,\infty)^d}
\Big(\prod_{j=1}^d r_j^{2\alpha_j}\Big)e^{-2nQ(r)}\prod_{j=1}^d r_j\,dr_j\,d\theta
=(2\pi)^d\int_{(0,\infty)^d} r^{2\alpha}\,e^{-2nQ(r)}\prod_{j=1}^d r_j\,dr_j,
\end{equation*}
where $r^{2\alpha}:=\prod_{j=1}^d r_j^{2\alpha_j}$ and $d\theta:=d\theta_1\cdots d\theta_d$.
Set $s_j=\log r_j$ so that $r_j=e^{s_j}$ and $r_jdr_j=e^{2s_j}ds_j$. Define
\begin{equation}\label{eq:qdef-final}
q(s):=Q(e^{s_1},\dots,e^{s_d}),\qquad s\in\R^d,
\end{equation}
and for $t\in\Sigma$ set
\begin{equation}\label{eq:F-def-final}
F_t(s):=\langle t,s\rangle-q(s).
\end{equation}
Then for all $n\ge n_0$ and all $\alpha\in\Z_+^d$ with $|\alpha|\le n$, setting $t:=\alpha/n$, we have
\begin{equation}\label{eq:Laplace-form-final}
I_{\alpha,n}=(2\pi)^d\int_{\R^d}\exp\!\big(2nF_t(s)\big)\,e^{2\langle\mathbf 1,s\rangle}\,ds,
\end{equation}
where $\mathbf 1=(1,\dots,1)$.

To control the integrand, we introduce the maximum coordinate function $m(s):=\max_{1\le j\le d}s_j$ and the threshold $M_0:=\log R$.

\begin{lem}\label{lem:qgrowth-final}
If $m(s)\ge M_0$, then $q(s)\ge (1+\varepsilon)\,m(s).$ Consequently, for every $t\in\Sigma$,
\begin{equation}\label{eq:F-growth-final}
F_t(s)\le -\varepsilon\,m(s)\qquad\text{whenever }m(s)\ge M_0,
\end{equation}
uniformly in $t\in\Sigma$.
\end{lem}

\begin{proof}
Choose $j\in\{1,\dots,d\}$ such that $s_j=m(s)$. Then $\|e^s\|\ge e^{s_j}=e^{m(s)}$, hence
$\log\|e^s\|\ge m(s)$. Since $m(s)\ge M_0=\log R$, we have $\|e^s\|\ge R$ and \eqref{eq:Q-growth-orthopoly} yields 
\begin{equation}\label{eq:qgrowth-final}
    q(s)=Q(e^s)\ge (1+\varepsilon)\log\|e^s\|\ge (1+\varepsilon)m(s).
\end{equation}

Now fix $t\in\Sigma$. Since $t_i\ge 0$ and $s_i\le m(s)$ for all $i$, we have 
\begin{equation*}
\langle t,s\rangle=\sum_{i=1}^d t_i s_i \le \sum_{i=1}^d t_i\,m(s)=|t|_1\,m(s).
\end{equation*}
As $|t|_1\le 1$ and $m(s)\ge M_0=\log R\ge 0$, we obtain $\langle t,s\rangle\le m(s)$.
Combining this with \eqref{eq:qgrowth-final} gives 
\begin{equation*}
F_t(s)=\langle t,s\rangle-q(s)\le m(s)-(1+\varepsilon)m(s)=-\varepsilon m(s),
\end{equation*}
which proves \eqref{eq:F-growth-final}. The estimate is clearly uniform in $t\in\Sigma$.
\end{proof}

Next, we introduce the candidate profile function and prove its basic regularity. Define $u:\Sigma\to\R\cup\{+\infty\}$ by
\begin{equation}\label{eq:u-def-final}
u(t):=\sup_{s\in\R^d}F_t(s)=\sup_{s\in\R^d}\big(\langle t,s\rangle-q(s)\big),\qquad t\in\Sigma.
\end{equation}

\begin{cor}\label{cor:u-finite-continuous-final}
The profile function $u(t)$ is finite everywhere on the simplex, meaning $u(t)\in\R$ for all $t\in\Sigma$, and $u\in \mathscr{C}(\Sigma)$.
\end{cor}

\begin{proof}
Fix $t\in\Sigma$. We show that $F_t$ is bounded above on $\R^d$.
If $m(s)\ge M_0$, then Lemma~\ref{lem:qgrowth-final} yields $F_t(s)\le -\varepsilon m(s)\le -\varepsilon M_0$.
If $m(s)\le M_0$, then $s_i\le M_0$ for all $i$, hence 
\begin{equation*}
\langle t,s\rangle=\sum_{i=1}^d t_i s_i \le \sum_{i=1}^d t_i\,M_0=|t|_1 M_0\le M_0,
\end{equation*}
and since $q(s)\ge 0$ we obtain $F_t(s)\le M_0$. Therefore $u(t)=\sup_{\R^d}F_t<\infty$.
Also $u(t)\ge F_t(0)=-q(0)>-\infty$, hence $u(t)\in\R$.

Moreover, $u$ is the supremum of affine functions of $t$, hence convex and lower semicontinuous on $\Sigma$.
By the Gale--Klee--Rockafellar theorem \cite[condition (L), p.~868 and Thm.~1, p.~869]{GKR1968},
a \emph{finite} convex function on a simplex is upper semicontinuous. Therefore $u$ is continuous on $\Sigma$.
\end{proof}

\begin{rem}\label{rem:convex-continuity-on-P}
The Gale--Klee--Rockafellar theorem extends from simplices to general polytopes. In contrast, for a general convex body $P\subset\R^d$, a finite, convex, and lower semicontinuous function $f:P\to\R$ is guaranteed to be continuous only on its relative interior $\mathrm{ri}(P)$ (see \cite[Thm.~10.1]{Rockafellar70}); discontinuities can occur on the boundary $\partial P$. This justifies our reliance on the polytope structure of the standard simplex $\Sigma$.
\end{rem}

Next, we prove a uniform Laplace principle on a fixed cube. Fix $M>0$ and set $K_M:=[-M,M]^d$. For $t\in\Sigma$, define the localized supremum
\begin{equation}\label{eq:uM-def}
u_M(t):=\sup_{s\in K_M}F_t(s).
\end{equation}

\begin{lem}\label{lem:laplace-on-cube-final}
With $u_M(t)$ defined as above, we have
\begin{equation}\label{eq:laplace-cube-final}
\lim_{n\to\infty}\ \sup_{t\in\Sigma}\left|
\frac1{2n}\log\!\int_{K_M} e^{2nF_t(s)}e^{2\langle\mathbf 1,s\rangle}\,ds-u_M(t)\right|=0.
\end{equation}
\end{lem}

\begin{proof}
Fix $t\in\Sigma$. Since $F_t$ is continuous and $K_M$ is compact, the supremum
$u_M(t)=\sup_{K_M}F_t$ is attained at some point $s_t\in K_M$.

Set $L_M:=\sup_{s\in K_M}\|\nabla q(s)\|<\infty$. For $s,y\in K_M$, by Cauchy--Schwarz and
$\|t\|_2\le |t|_1\le 1$ we have
\begin{equation*}
|F_t(s)-F_t(y)| \le |\langle t,s-y\rangle|+|q(s)-q(y)| \le \|s-y\|+L_M\|s-y\| =(L_M+1)\|s-y\|.
\end{equation*}
Fix $\delta>0$ and set $\rho:=\delta/(L_M+1)$. Then for every $s\in B(s_t,\rho)\cap K_M$, we have the lower bound $F_t(s)\ge F_t(s_t)-\delta=u_M(t)-\delta$.
Moreover, for $s\in K_M$ we have $\langle \mathbf 1,s\rangle\in[-dM,dM]$, hence $e^{2\langle\mathbf 1,s\rangle}\in[e^{-2dM},e^{2dM}]$ on $K_M$.

Since the map $y\mapsto \mathrm{Vol}(B(y,\rho)\cap K_M)$ is strictly positive and continuous on the compact set $K_M$, its infimum $c_{M,\rho}:=\inf_{y\in K_M}\mathrm{Vol}\big(B(y,\rho)\cap K_M\big)$ is strictly positive. Consequently, we obtain the lower bound:
\begin{equation*}
\int_{K_M} e^{2nF_t(s)}e^{2\langle\mathbf 1,s\rangle}\,ds \ge \int_{B(s_t,\rho)\cap K_M} e^{2nF_t(s)}e^{2\langle\mathbf 1,s\rangle}\,ds \ge e^{-2dM}\,c_{M,\rho}\,e^{2n(u_M(t)-\delta)}.
\end{equation*}
On the other hand, using $F_t\le u_M(t)$ on $K_M$, we have the following upper bound:
\begin{equation*}
\int_{K_M} e^{2nF_t(s)}e^{2\langle\mathbf 1,s\rangle}\,ds \le e^{2dM}\,\mathrm{Vol}(K_M)\,e^{2n u_M(t)}.
\end{equation*}
Taking $\frac1{2n}\log$ of these bounds gives
\begin{align*}
u_M(t)-\delta+\frac1{2n}\log\!\big(e^{-2dM}c_{M,\rho}\big)
&\le \frac1{2n}\log\!\int_{K_M} e^{2nF_t(s)}e^{2\langle\mathbf 1,s\rangle}\,ds \\
&\le u_M(t)+\frac1{2n}\log\!\big(e^{2dM}\mathrm{Vol}(K_M)\big).
\end{align*}
Since the logarithmic terms are independent of $t$, they vanish uniformly as $n\to\infty$. Letting $n\to\infty$ and then $\delta\downarrow 0$ yields \eqref{eq:laplace-cube-final}.
\end{proof}

\begin{lem}\label{lem:Dini-uM-final}
As $M\to\infty$, $u_M(t)\uparrow u(t)$ for each $t\in\Sigma$. Moreover, the convergence is uniform
\begin{equation}\label{eq:Dini-uM-final}
\lim_{M\to\infty}\ \sup_{t\in\Sigma}|u_M(t)-u(t)|=0.
\end{equation}
\end{lem}

\begin{proof}
Since $K_M\subset K_{M'}$ for $M\le M'$ and $\bigcup_{M\ge 1}K_M=\R^d$, we have $u_M(t)\uparrow u(t)$ for each $t\in\Sigma$.

Fix $M\in\mathbb{N}$. For any $t,t'\in\Sigma$, we have
\begin{align*}
u_M(t)-u_M(t') &= \sup_{s\in K_M}F_t(s)-\sup_{s\in K_M}F_{t'}(s) \le \sup_{s\in K_M}\big(F_t(s)-F_{t'}(s)\big) \\
&= \sup_{s\in K_M}\langle t-t',s\rangle \le \sqrt d\,M\,\|t-t'\|_2,
\end{align*}
where we used Cauchy--Schwarz and the bound $\|s\|_2\le \sqrt d\,M$ for $s\in K_M$.
The same estimate holds with $t$ and $t'$ interchanged, hence $u_M$ is continuous on $\Sigma$.
Since $u$ is continuous on $\Sigma$ by Corollary~\ref{cor:u-finite-continuous-final} and $\Sigma$ is compact,
Dini's theorem applied to the increasing sequence $(u_M)_{M\in\mathbb{N}}$ yields \eqref{eq:Dini-uM-final}.
\end{proof}

We now evaluate the global integral by decomposing the domain. For $M\ge M_0$, we write $\R^d$ as the disjoint union $\R^d = A_M \cup B_M$, where
\begin{equation}\label{eq:domain-decomp}
A_M := (-\infty, M]^d = \{s\in\R^d:\ m(s)\le M\} \quad \text{and} \quad B_M := \{s\in\R^d:\ m(s)>M\}.
\end{equation}
We first provide a uniform bound for the integral over the tail region $B_M$.

\begin{lem}[Uniform tail estimate]\label{lem:tail-estimate-final}
For all $n$ with $n\varepsilon>d$ and all $t\in\Sigma$, we have
\begin{equation}\label{eq:tail-bound-final}
\int_{B_M} e^{2nF_t(s)}e^{2\langle\mathbf 1,s\rangle}\,ds
\le \frac{d\,2^{-(d-1)}}{2n\varepsilon-2d}\,e^{(2d-2n\varepsilon)M}.
\end{equation}
In particular, $\limsup_{n\to\infty}\ \sup_{t\in\Sigma}\frac1{2n}\log\!\int_{B_M} e^{2nF_t(s)}e^{2\langle\mathbf 1,s\rangle}\,ds
\le -\varepsilon M.$

\end{lem}

\begin{proof}
For each $j=1,\dots,d$ set $B_{M,j}:=\{s\in\R^d:\ s_j>M \text{ and } s_i\le s_j \text{ for all } i=1,\dots,d\}.$ If $s\in B_M$, choose $j$ such that $s_j=m(s)>M$. Then $s\in B_{M,j}$, hence
$B_M\subset\bigcup_{j=1}^d B_{M,j}$. Moreover, $m(s)=s_j$ on $B_{M,j}$. Therefore,
\begin{equation}\label{eq:UnionBd}
    \int_{B_M} e^{2nF_t(s)}e^{2\langle\mathbf 1,s\rangle}\,ds
\le \sum_{j=1}^d \int_{B_{M,j}} e^{2nF_t(s)}e^{2\langle\mathbf 1,s\rangle}\,ds.
\end{equation}

Fix $t\in\Sigma$ and $s\in B_{M,j}$. Since $s_i\le s_j$ for all $i$ and $t_i\ge 0$ with $|t|_1\le 1$, we have 
\begin{equation*}
\langle t,s\rangle=\sum_{i=1}^d t_i s_i \le \sum_{i=1}^d t_i s_j = |t|_1 s_j \le s_j.
\end{equation*}
Also $s_j=m(s)>M\ge M_0$, so Lemma~\ref{lem:qgrowth-final} yields $q(s)\ge (1+\varepsilon)m(s)=(1+\varepsilon)s_j$.
Hence $F_t(s)=\langle t,s\rangle-q(s)\le -\varepsilon s_j$, and consequently $e^{2nF_t(s)}e^{2\langle\mathbf 1,s\rangle} \le e^{-2n\varepsilon s_j}\,e^{2(s_1+\cdots+s_d)}$.

Since the integrand is nonnegative, by Tonelli's theorem, we may integrate iteratively. Writing $\tau:=s_j$, we obtain
\begin{align*}
\int_{B_{M,j}} e^{2nF_t(s)}e^{2\langle\mathbf 1,s\rangle}\,ds
&\le \int_{\tau=M}^{\infty} e^{-2n\varepsilon \tau}
\left(\int_{\{s_i\le \tau,\ i\neq j\}} e^{2\sum_{i\neq j}s_i}\,ds_{i\neq j}\right)
e^{2\tau}\,d\tau \\
&= \int_M^\infty e^{-2n\varepsilon \tau}
\left(\prod_{i\neq j}\int_{-\infty}^\tau e^{2s_i}\,ds_i\right)e^{2\tau}\,d\tau \\
&= 2^{-(d-1)}\int_M^\infty e^{-2n\varepsilon \tau}\,e^{2(d-1)\tau}\,e^{2\tau}\,d\tau \\
&= 2^{-(d-1)}\int_M^\infty e^{(2d-2n\varepsilon)\tau}\,d\tau
= \frac{2^{-(d-1)}}{2n\varepsilon-2d}\,e^{(2d-2n\varepsilon)M},
\end{align*}
for all $n\varepsilon>d$. Summing over $j=1,\dots,d$ and using \eqref{eq:UnionBd} yields \eqref{eq:tail-bound-final}.

Finally, taking $\frac1{2n}\log$ in \eqref{eq:tail-bound-final} gives
\begin{equation*}
\frac1{2n}\log\!\int_{B_M} e^{2nF_t(s)}e^{2\langle\mathbf 1,s\rangle}\,ds
\le \frac1{2n}\log\!\Big(\frac{d\,2^{-(d-1)}}{2n\varepsilon-2d}\Big) +\Big(\frac{d}{n}-\varepsilon\Big)M,
\end{equation*}
uniformly in $t\in\Sigma$. Letting $n\to\infty$ gives the last statement.
\end{proof}

We now combine  Lemma~\ref{lem:laplace-on-cube-final}, Lemma~\ref{lem:Dini-uM-final}, and Lemma~\ref{lem:tail-estimate-final} to obtain a uniform Laplace principle on $\R^d$. Hence the exponential profile for $(c_{\alpha,n})$.

\begin{prop}[Exponential profile for $(c_{\alpha,n})$]\label{prop:exp-profile-orthonormal-final}
Let $u$ be defined by \eqref{eq:u-def-final}. Then
\begin{equation}\label{eq:I-asymptotics-final}
\lim_{n\to\infty}\ \sup_{\alpha\in n\Sigma\cap\Z_+^d}\left|
\frac1{2n}\log I_{\alpha,n}-u(\alpha/n)\right|=0.
\end{equation}
Consequently, $\lim_{n\to\infty}\ \sup_{\alpha\in n\Sigma\cap\Z_+^d}\left|
\frac1n\log c_{\alpha,n}+u(\alpha/n)\right|=0.$
\end{prop}

\begin{proof}
Fix $M\ge M_0$. Using the decomposition $\R^d = A_M \cup B_M$, for $t\in\Sigma$ we define
\begin{align*}
J_n(t):=\int_{\R^d} e^{2nF_t(s)}\,e^{2\langle\mathbf 1,s\rangle}\,ds 
\phantom{:}=\int_{A_M} e^{2nF_t(s)}e^{2\langle\mathbf 1,s\rangle}\,ds+\int_{B_M} e^{2nF_t(s)}e^{2\langle\mathbf 1,s\rangle}\,ds.
\end{align*}
By \eqref{eq:Laplace-form-final}, for $n\ge n_0$ and $t=\alpha/n$ we have $I_{\alpha,n}=(2\pi)^d\,J_n(t)$. Since $F_t(s)\le u(t)$ for all $s\in\R^d$, we obtain
\begin{equation}\label{eq:Jnt-AM-upper-final-rev}
\int_{A_M} e^{2nF_t(s)}e^{2\langle\mathbf 1,s\rangle}\,ds
\le e^{2nu(t)}\int_{A_M} e^{2\langle\mathbf 1,s\rangle}\,ds
= e^{2nu(t)}\Big(\frac{e^{2M}}{2}\Big)^d .
\end{equation}
Moreover, for $n\ge n_0$ we have $n\varepsilon>d$, and Lemma~\ref{lem:tail-estimate-final} yields
\begin{equation}\label{eq:Jnt-BM-upper-final-rev}
\int_{B_M} e^{2nF_t(s)}e^{2\langle\mathbf 1,s\rangle}\,ds
\le \frac{d\,2^{-(d-1)}}{2n\varepsilon-2d}\,e^{(2d-2n\varepsilon)M}.
\end{equation}
Combining \eqref{eq:Jnt-AM-upper-final-rev}--\eqref{eq:Jnt-BM-upper-final-rev} and using the standard inequality $a+b\le 2\max\{a,b\}$ gives
\begin{equation}\label{eq:Jnt-upper-max-final}
\frac1{2n}\log J_n(t)\le \max\{u(t),-\varepsilon M\}+o_{n,M}(1),
\end{equation}
where $o_{n,M}(1)\to 0$ as $n\to\infty$, uniformly in $t\in\Sigma$ (for fixed $M$).
Since $u(t)\ge F_t(0)=-q(0)$ for all $t\in\Sigma$, choosing $M$ sufficiently large so  that $-\varepsilon M<-q(0)$, gives
\begin{equation}\label{eq:Jnt-limsup-final}
\limsup_{n\to\infty}\ \sup_{t\in\Sigma}\Big(\frac1{2n}\log J_n(t)-u(t)\Big)\le 0.
\end{equation}

On the other hand, with $K_M=[-M,M]^d\subset A_M$ we have $J_n(t)\ge \int_{K_M} e^{2nF_t(s)}e^{2\langle\mathbf 1,s\rangle}\,ds$.
Recall that $u_M(t)=\sup_{s\in K_M}F_t(s)$. Then Lemma~\ref{lem:laplace-on-cube-final} implies
\begin{equation}\label{eq:Jnt-liminf-uM-final}
\liminf_{n\to\infty}\ \inf_{t\in\Sigma}\Big(\frac1{2n}\log J_n(t)-u_M(t)\Big)\ge 0.
\end{equation}
Consequently, for each fixed $M$, we have
\begin{equation*}
\liminf_{n\to\infty}\ \inf_{t\in\Sigma}\Big(\frac1{2n}\log J_n(t)-u(t)\Big)
\ge -\sup_{t\in\Sigma}\big(u(t)-u_M(t)\big).
\end{equation*}
Letting $M\to\infty$ and using \eqref{eq:Dini-uM-final} gives
\begin{equation}\label{eq:Jnt-liminf-final}
\liminf_{n\to\infty}\ \inf_{t\in\Sigma}\Big(\frac1{2n}\log J_n(t)-u(t)\Big)\ge 0.
\end{equation}

Combining \eqref{eq:Jnt-limsup-final} and \eqref{eq:Jnt-liminf-final} yields
\begin{equation}\label{eq:Jnt-unif-final}
\lim_{n\to\infty}\ \sup_{t\in\Sigma}\left|\frac1{2n}\log J_n(t)-u(t)\right|=0.
\end{equation}
Since $I_{\alpha,n}=(2\pi)^dJ_n(\alpha/n)$ for $n\ge n_0$, the constant contributes $\frac1{2n}\log(2\pi)^d\to 0$,
and \eqref{eq:I-asymptotics-final} follows. Finally, since $c_{\alpha,n}=I_{\alpha,n}^{-1/2},$ we are done.
\end{proof}

\begin{rem}\label{rem:general-P-final}
The same strategy extends to any polytope $P\subset\R_+^d$.
Set $\lambda_P:=\sup_{t\in P}|t|_1<\infty$.
Assume that there exist $\varepsilon>0$ and $R>0$ such that
\begin{equation*}
Q(z)\ge (\lambda_P+\varepsilon)\log\|z\|
\qquad\text{for }\|z\|\ge R.
\end{equation*}
Then the growth and tail estimates remain uniform for $t\in P$, and the proof of the uniform Laplace principle
and the exponential profile carries over with $P$ in place of $\Sigma$.
We do not pursue the details here.
\end{rem}
\medskip
Note that Proposition~\ref{prop:exp-profile-orthonormal-final} shows that the orthonormal ensemble \eqref{eq:random-orthopoly-ensemble} satisfies our exponential profile condition \eqref{eq:profile} on the standard unit simplex $\Sigma$, with the profile function $u:\Sigma \to \R$ given by $u(t)=\sup_{s\in\R^d}\big(\langle t,s\rangle-q(s)\big)$, where $q(s):=Q(e^s)$. Equivalently, this corresponds to the continuous weight function $f:\Sigma \to (0, \infty)$ defined by $f(t)=e^{-u(t)}$. Hence, the limiting zero current is given by $dd^c \Phi_{\Sigma,f}$. In this way, we recover the equidistribution result in \cite[Thm.~1.1]{Bay19} as a special case. 

We now show that the limiting potential $\Phi_{\Sigma,f}$ coincides with the weighted extremal function associated to $Q$. Recall that the weighted extremal function associated to $Q$ is defined by
\begin{equation*}
V_Q(z) := \sup\{w(z) : w \in \mathcal{L}(\C^d), \ w \le Q \text{ on } \C^d\}.
\end{equation*}
Seminal results of Siciak and Zakharyuta (see, e.g., the appendix by T. Bloom in \cite{SaffTotik97}) imply that $V_Q \in \mathcal{L}^+(\C^d)$ and
\begin{equation}\label{eq:VQ-Siciak}
V_Q(z) = \sup\left\{ \frac{1}{\deg p} \log |p(z)| : p \text{ is a polynomial and } \sup_{\zeta\in\C^d} \big(|p(\zeta)|e^{-\deg(p)Q(\zeta)}\big) \le 1\right\}.
\end{equation}
Moreover, a result of Berman \cite[Proposition 2.1]{Ber09} implies that $V_Q$ is of class $\mathscr{C}^{1,1}$.

\begin{prop}\label{prop:Phi-equals-VQ}
Assume that $Q$ is torus-invariant and satisfies the growth condition \eqref{eq:Q-growth-orthopoly}. Let $V_Q$ be the weighted extremal function defined above. Then $\Phi_{\Sigma,f}=V_Q$ on $\C^d$, and consequently $dd^c\Phi_{\Sigma,f}=dd^cV_Q$.
\end{prop}

\begin{proof}
Let $s=\Log |z|\in\R^d$ for $z\in(\C^*)^d$, and set $q(s):=Q(e^s)$. Since $Q$ is torus-invariant, $q$ is a well-defined function on $\R^d$. Let $q^*(t):=\sup_{y\in\R^d}\big(\langle t,y\rangle-q(y)\big)$ denote its Legendre--Fenchel transform. By construction, the restriction $q^*|_{\Sigma}$ coincides with the profile function $u$ from \eqref{eq:u-def-final}. 

To establish $\Phi_{\Sigma,f}\le V_Q$, we observe that for any $t\in\Sigma$ and $s\in\R^d$, the definition of $q^*$ gives $\langle t,s\rangle-q^*(t)\le q(s)$. Taking the supremum over $t\in\Sigma$ and using $q^*|_{\Sigma}=u$, we obtain $\Phi_{\Sigma,f}(e^s)\le q(s)$, which means $\Phi_{\Sigma,f}(z)\le Q(z)$ on $(\C^*)^d$. Since $Q$ is continuous and $\Phi_{\Sigma,f}$ is defined via upper semicontinuous extension from the dense torus, this bound extends to $\C^d$. As we already know that $\Phi_{\Sigma,f}\in\mathcal{L}(\C^d)$, by maximality of $V_Q$ we have that $\Phi_{\Sigma,f}\le V_Q$ on $\C^d$.

Conversely, let $v(s):=V_Q(e^s)$. Since $Q$ and the Lelong class $\mathcal{L}(\C^d)$ are  torus-invariant, the extremal function $V_Q$ is a torus-invariant plurisubharmonic function. Thus, $v$ is a finite convex function on $\R^d$ satisfying $v \le q$. Convex duality implies $v^*(t)\ge q^*(t)$ for all $t\in\R^d$. Furthermore, the assumption $V_Q \in\mathcal L(\C^d)$ implies there exists a constant $C_0$ such that $v(s)\le \max\{0,s_1,\dots,s_d\}+C_0 = h_{\Sigma}(s) + C_0$, where $h_{\Sigma}(s) := \sup_{\tau\in\Sigma}\langle \tau,s\rangle$ is the support function of $\Sigma$. If $t\notin\Sigma$, strict separation of the closed convex set $\Sigma$ from the point $t$ provides $s_0\in\R^d$ such that $\langle t,s_0\rangle > h_{\Sigma}(s_0)$. For $\lambda>0$, using $v(\lambda s_0)\le \lambda h_{\Sigma}(s_0)+C_0$, we obtain
\begin{equation*}
\langle t,\lambda s_0\rangle - v(\lambda s_0) \ge \lambda\big(\langle t,s_0\rangle-h_{\Sigma}(s_0)\big)-C_0 \xrightarrow[\lambda\to\infty]{} +\infty.
\end{equation*}
Hence $v^*(t)=+\infty$ for all $t\notin\Sigma$.

Because $v$ is a finite continuous convex function on $\R^d$, applying the Fenchel--Moreau theorem ($v = v^{**}$) yields
\begin{equation*}
v(s)=\sup_{t\in\R^d}\big(\langle t,s\rangle-v^*(t)\big)=\sup_{t\in\Sigma}\big(\langle t,s\rangle-v^*(t)\big) \le \sup_{t\in\Sigma}\big(\langle t,s\rangle-u(t)\big) = \Phi_{\Sigma,f}(e^s).
\end{equation*}
This shows $V_Q(z)\le \Phi_{\Sigma,f}(z)$ for $z \in (\C^*)^d$. Since $V_Q$ is continuous and $\Phi_{\Sigma,f}$ is defined via upper semicontinuous extension from the dense torus, this inequality extends to $\C^d$, completing the proof.
\end{proof}

\end{document}